\documentclass[12pt]{amsart}
\usepackage{amsmath,amssymb}
\usepackage{amsfonts}
\usepackage[dvips]{graphicx}
\usepackage[latin1]{inputenc}
\usepackage{multicol}
\usepackage{epsfig}
\usepackage[all]{xy}

\newtheorem{theorem}{Theorem}[section] 
\newtheorem{lemma}[theorem]{Lemma}     

\theoremstyle{definition}

\newcommand{\im}{\operatorname{Im}}

\newcommand\alge{\mathfrak{A}}

\newcommand\oink{\mathcal O}

\newcommand\Da{\mathfrak{D}}

\newcommand\Ea{\mathfrak{E}}

\newcommand\qdef{\eth}

\newcommand\bmattrix[4]{\left(\begin{array}{cc}#1&#2\\#3&#4\end{array}\right)}
\newcommand\vvect[2]{\left(\begin{array}{c}#1\\#2\end{array}\right)}
\newcommand\sbmattrix[4]{\textnormal{\scriptsize$\left(\begin{array}{cc}#1&#2\\#3&#4\end{array}\right)$\normalsize}}

\newcommand\matrici{\mathbb{M}}

\begin{document}

\title{On the missing branches of the Bruhat-Tits tree.}

\author{Luis Arenas-Carmona\\
Claudio Bravo}

\maketitle

\begin{abstract}
Let $k$ be a local field and let $\alge$ be the two-by-two matrix algebra over $k$.
In our previous work we developped a theory that allows the computation of the set of maximal orders in $\alge$
containing a given suborder. This set is given as a subtree of the Bruhat-Tits tree that is called the branch
 of the order. Branches have been used to study the global selectivity problem and also to compute local
embedding numbers. They can usually be described in terms of two invariants. To compute these
invariants explicitely, the strategy in our past work has been
visualizing branches through the explicit representation
of the Bruhat-Tits tree in terms of balls in $k$. This is easier for orders spanning a split commutative sub-algebra,
i.e., an algebra isomorphic to $k\times k$. 
In the present work, we develop a theory of branches over field extension that can be used
to extend our previous computations to orders spanning a field. We use the same idea to compute
branches for orders generated by arbitrary pairs of non-nilpotent pure quaternions. In fact, the hypotheses 
on the generators are not essential.
\end{abstract}

\section{Introduction}

Let $\Omega$ be an order in the local matrix algebra $\alge=\mathbb{M}_2(k)$, where $k$ is a local field. The
set $S(\Omega)$ of maximal orders in  $\alge$ containing $\Omega$ plays a significant role in the study of 
several interesting
arithmetical phenomena, like the selectivity problem, determining whether a global order embeds into all or just into
some of the maximal orders in a quaternion algebra over a global field $F$ \cite{Eichler2}, 
\cite{osbotbtt}, or describing the normalizers in 
$\mathrm{PSL}_2(F)$ of Eichler orders and congruence subgroups \cite{scffgeo}. The study of this set plays also
a significant role in determining quotient graphs for some arithmetically important subgroups of the 
general linear group \cite{cqqgvro}.

Usually, we describe this set of orders
as the vertex set of a subgraph $\mathfrak{s}(\Omega)$ of the Bruhat-Tits tree
$\mathfrak{t}(k)$, a tree
whose vertices are the maximal orders in $\alge$, while two of them $\Da$ and $\Da'$ are neighbors if
in some basis they have the form 
$$\Da=\bmattrix{\oink}{\oink}{\oink}{\oink}\textnormal{ and }
\Da'=\bmattrix{\oink}{\pi^{-1}\oink}{\pi\oink}{\oink},$$
where $\oink$ is the ring of integers, and $\pi$ is a uniformizing parameter of $k$. The graph 
 $\mathfrak{s}(\Omega)$, which we call the branch of $\Omega$, is usually a tubular neighborhood of a path
$\mathfrak{p}$, i.e., a thick path with stem $\mathfrak{p}$ as defined in \S2,
 except for a couple of very simple orders
(c.f. \cite[Prop 5.3]{Eichler2} and \cite[Prop 5.4]{Eichler2}):
\begin{enumerate}
\item $\Omega=\oink$, identified with the ring of scalar matrices with integral entries.
\item $\Omega=\oink[u]$ where $u\in\alge\backslash\{0\}$ is nilpotent. This is called an idempotent order.
\end{enumerate}
In the latter case, the branch of $\Omega$ is a graph called an infinite leaf \cite{abelianos},
 and can be seen as
the limit of a sequence of thick paths, with a common vertex of valency one,
 whose stems get infinitely far away. This type of set is called a horoball in some previous
 literature on diophantine aproximation \cite{Paulin}.

The graph $\mathfrak{t}(k)$ can alternatively be defined as a graph whose vertices are the balls in $k$,
while two balls are neighbors if one is a maximal proper sub-ball of the other. This observation becomes
a powerful tool to compute branches, while embeddings can be easily interpreted graphically in this
context, provided that the order $\Omega$ is the intersection of a family of maximal orders.
This property allowed us to compute local embedding numbers for orders spanning algebras
whose maximal semisimple quotient is the split commutative algebra $k\times k$ \cite{omeat}. 
Unfortunately, this is not the case for many interesting orders, like those contained in maximal subfields.

The purpose of the present work is to develope a technique that can be applied to the latter orders. 
With this in mind, we construct an embedding of the graph $\mathfrak{t}(k)$, or more precisely an appropriate
subdivision of it, into the graph $\mathfrak{t}(L)$ for any finite field extension $L/k$. In order to show
the scope of this new tool, we extend two of our previous results.  The first one extend the explicit
formulas obtained in \cite{osbotbtt} to compute the invariant describing the branch $\mathfrak{s}(\Omega)$,
when $\Omega=\oink[i,j]$ is the order generated by two orthogonal pure quaternions,
i.e., two matrices satisfying the relations
\begin{equation}\label{133}
i^2=a,\qquad j^2=b,\qquad ij+ji=0,
\end{equation}
which are the standard generators
 of a quaternion algebra, and therefore play a central role in the theory.
However, orders generated by more general pairs of pure quaternions do apear naturally in practical problems,
making desirable to extend this computation in a more general setting (c.f. \cite{Brzezinski}).
 More precisely, in this work
we no longer require the orthogonality condition.

 Our second result extends the embedding number
computations in \cite{omeat} to the case of orders contained in fields, the only orders of non-maximal rank that
failed to be considered in our previous work.
$$ $$

\paragraph{\textbf{Conventions on graphs and walks}}
In all that follows, a graph $\mathfrak{g}$ is a set of vertices $V_\mathfrak{g}$ toghether with a
symmetric relation called the neighborhood relation in $\mathfrak{g}$. A subgraph of   $\mathfrak{g}$
is any graph $\mathfrak{h}$, satisfying $V_\mathfrak{h}\subseteq V_\mathfrak{g}$ and
 whose  neighborhood relation implies the induced relation. If the
neighborhood relation in $\mathfrak{h}$ is the induced relation, we call it a full subgraph.
We are not concerned with non-full subgraphs in this work.
The intersection of a family of full subgraphs is also a full subgraph with the natural conventions.
The valency of a vertex, in a given graph $\mathfrak{g}$,  is the number of its neighbors.
Vertices of valency $1$ are called optimal, since, when $\mathfrak{g}=\mathfrak{s}(\Omega)$
as before, they correspond to maximal orders in which $\Omega$ is optimal \cite{omeat}.
A finite walk  in $\mathfrak{g}$ is a sequence of vertices $v_0v_1\dots v_r$ satisfying the following conditions:
\begin{enumerate}
\item Each par of consecutive vertices are neighbors.
\item There is no backtracking, i.e., $v_i\neq v_{i+2}$ for every $i=0,\dots,r-2$.
\end{enumerate}
We usually emphasize the initial and last vertex in the walk by saying a walk from $v_0$ to $v_r$.
A graph $\mathfrak{g}$ is connected if there is a walk from every certex $v_0\in V_\mathfrak{g}$
to every vertex $v_r\in V_\mathfrak{g}$.
A cycle is a walk $v_0v_1\dots v_r$ satisfying $v_r=v_0$. A tree is a connected graph with no cycles.
Equivalently, a graph is a tree if there is a unique walk from $v_0$ to $v_r$ for any pair 
$(v_0,v_r)\in V_\mathfrak{g}\times V_\mathfrak{g}$. A walk in a tree has no repeated vertices.
All graph considered here are trees.  We call $r$ the length of the walk $v_0\dots v_r$,
and we admit the walk $v_0$ of length $0$. The tree distance in $\mathfrak{g}$ is the
metric $\delta$  in  $V_\mathfrak{g}$ defined by $\delta(v,w)=r$ if the walk from $v$ to $w$ has length $r$. 
We also consider two types of infinite walks:
\begin{enumerate}
\item An infinite walk is a sequence of the form $v_0v_1\dots$ with one vertex for each natural number.
\item A double infinite walk is a sequence of the form $\dots v_{-1}v_0v_1\dots$ 
with one vertex for each integer.
\end{enumerate}
We identify the double infinite walks $\dots v_{-1}v_0v_1\dots$ and $\dots v'_{-1}v'_0v'_1\dots$ provided
$v'_t=v_{t+m}$ for a fixed integer $m$ and every integer $t$. We also define an equivalence relation between 
infinite walks, where $v_0v_1\dots$ and $v'_0v'_1\dots$ are related whenever
$v'_t=v_{t+m}$ for a fixed integer $m$ and every big enough integer $t$. Equivalence classes
in the latter sense are called ends. We usually represent an end graphically by a dot beyond the walk.
Furthermore, for any subgraph $\mathfrak{h}$ of $\mathfrak{g}$, there is a natural embedding
from the set of ends of $\mathfrak{h}$ to the set of ends of $\mathfrak{g}$. We identify the
ends of $\mathfrak{h}$ with the corresponding ends of $\mathfrak{g}$, and usually write expresions
like the end $a$ is in $\mathfrak{h}$,
 or belongs to $\mathfrak{h}$, in this sense, and even write $a\in\mathfrak{h}$.

\section{Main results}
 
In all that follows, $k$, $\oink$, $\alge$ and $\Omega$ are as before. Let $\pi$ be a uniformizing parameter
in $k$, and let $\nu:k^*\rightarrow\mathbb{Z}$ be the usual valuation, normalized in a way that $\nu(\pi)=1$.
We let $N$ be the number of quadratic classes of ramified units in $\oink^*$, i.e., units  whose square
 roots generate ramified extensions, so that the set of square classes in  $k^*$ is 
$$k^*/k^{*2}=\{\bar{1},\bar{\Delta},\bar{u}_1,\dots,\bar{u}_N,\bar{\pi}_1,\dots,\bar{\pi}_{N+2}\},$$
where $\Delta$ is a unit of minimal quadratic defect \cite{Om}, $u_1,\dots,u_N$ are ramified units,
and $\pi_1,\dots,\pi_{N+2}$ are uniformizing parameters. Our results are usually stated in terms of the quadratic
defect $\qdef$. For any element $a\in k^*$, the quadratic defect is the smallest fractional ideal in $k$
spanned by an element of the form $a-b^2$ (c.f. \cite{Om}).
 For the square class representatives shown above,
this is computed as follows:
 $$\qdef(1)=\{0\},\qquad \qdef(\Delta)=(4),\qquad
\qdef(\pi_n)=(\pi),\qquad \qdef(u_n)=(\pi^{2s+1})$$
for some integer $s=s(u_n)$ satisfying $0\leq s<\nu(2)$.

Let $\mathfrak{t}=\mathfrak{t}(k)$ be the Bruhat-Tits tree, with the tree-distance $\delta$ defined in \S1.
 For every vertex $v$, we define the ball of radius $p$ around $v$, as the full subgraph 
$\mathfrak{b}=\mathfrak{b}_v[p]$ whose vertex set satisfies
$V_{\mathfrak{b}}=\{w\in V_{\mathfrak{t}}| \delta(v,w)\leq p\}$.
  For every walk $w=\dots v_{i-1}v_iv_{i+1}\dots$ in $\mathfrak{t}$,
finite or not, we call the interval $I(w)=\{\dots,i-1,i,i+1,\dots\}\subseteq\mathbb{Z}$ its index interval.
A thick path is a full subgraph of the form 
$\mathfrak{s}=\bigcup_{n\in I(w)}\mathfrak{b}_{v_n}[p]$, where $w=\dots v_{i-1}v_iv_{i+1}\dots$ is a walk.
The integer $p\leq0$ is called the depth of $\mathfrak{s}$, while 
$\mathfrak{m}=\{v_i|i\in I(w)\}$ is called the stem.
Note that $\mathfrak{s}=\mathfrak{m}$ if $p=0$. 
As mentioned in \S1,  the branch $\mathfrak{s}(\Omega)$ is a thick path for most orders $\Omega$.
Many combinatorial properties of the branch
 can be described in terms of two invariants. The stem length $l$, i.e., the length of $w$,
and  the depth $p$. The computation of local embedding numbers \cite{omeat}
 or representation fields for global orders \cite{Eichler2}, \cite{cqqgvro}, 
reduces to determining these invariants.
This was done explicitly in \cite{osbotbtt},
 for orders generated by a pair of orthogonal pure quaternions as in (\ref{133}).
In fact, we already gave in \cite[Prop. 2.3]{osbotbtt} and \cite[Prop. 2.4]{osbotbtt} 
a method to do this in full generality, provided that the
relative position between the branches is known.  
The latter piece of data was collected, for orders generated by orthogonal pure quaternions,
  by thickenning the stem of the thick path corresponding to either generator, 
until we found a minimal setting where both branches do intersect, or equivalently, there exists a maximal order
containing each of the corresponding orders.
This kind of computations could be extended with enough work, but the method shown here is far simpler. 
It consist in giving a precise location for the stem of an order in terms of the simmetric product
$ij+ji\in K$. One this is known, the invariants can be computed by the results in \cite{osbotbtt}.

\begin{theorem}\label{t21}
 Let $i,j \in \mathbb{M}_2(k)$ be pure quaternion satisfying  $i^2=\alpha$, $j^2=\beta$ and $ij+ji=2\lambda$. Assume $\alpha$ and $\beta$ belong to a set of representatives of the form
$$Q=\{1,\Delta,u_1,\dots,u_N,\pi_1,\dots,\pi_{N+2}\},$$
 of all square classes, where $\Delta$ is a unit of minimal quadratic defect, 
$\{u_1,\dots,u_N\}$ is a set of representatives of all ramified 
units, while $\{\pi_1,\dots,\pi_{N+2}\}$ is a set of representatives of all uniformizing parameters. Let $d_f$
be the function defined case by case as follows:
\begin{enumerate}
\item If $\alpha,\beta\in\{1,\Delta\}$, then $d_f=  -\frac{1}{2}\nu \left( \frac{\lambda^2-\alpha\beta}{4} \right)$.
\item If $\alpha\in\{1,\Delta\}$, while $\beta\notin\{1,\Delta\}$ and $\qdef(\beta)=(\pi^{2t+1})$, then
 $d_f=t-\frac{1}{2}\nu \left( \lambda^2-\alpha\beta \right)$.
\item If $\alpha\notin\{1,\Delta\}$ and $\qdef(\alpha)=(\pi^{2s+1})$, while $\beta\in\{1,\Delta\}$, then
 $d_f=s-\frac{1}{2}\nu \left( \lambda^2-\alpha\beta \right)$.
\item If $\{\alpha,\beta\}\cap\{1,\Delta\}=\emptyset$, while $\qdef(\alpha)=(\pi^{2s+1})$
and $\qdef(\beta)=(\pi^{2t+1})$, then
$d_f=s+t-\frac{1}{2}\nu \big(4(\lambda^2-\alpha\beta) \big)$.
\end{enumerate}
Then if $d_f>0$, it equals the distance between the stems, otherwise the length of the intersection is
$$\mathrm{min}\{-2d_f,l(i),l(j)\},$$
where $l(q)$ is the stem length of $q$, which is $0$, $1$ or $\infty$ \cite{Eichler2}.
\end{theorem}

In what follows, the possibly negative function $d_f$ is refered to as the fake distance.
It is a distance only when it is non-negative.

Recall that an embedding $\phi:\Omega\rightarrow\Da$ is called optimal if $\hat{\phi}^{-1}(\Da)=\Omega$,
where $\hat{\phi}:k\otimes_\oink\Omega\rightarrow \matrici_2(k)$ is the natural extension. A suborder
$\Omega\subseteq\Da$ is optimal when the inclusion is an optimal embedding.
For any quadratic extension $L/k$, any order $\Omega=\oink_L^{\{t\}}$ of maximal rank in $L$, 
where $\oink_L$ is the ring of integers in $L$, and any  Eichler order $\Ea\subseteq\matrici_2(k)$
of level $r$, we let $X$ be the set of optimal embeddings  $\phi:\Omega\rightarrow\Ea$ and let $Y$
 be the set of optimal suborders of $\Ea$ that are isomorphic to  $\Omega$. 
We also let $\Gamma_1=k^*\Ea^*$, where $A^*$ denotes the group of units of a ring $A$,
 and let $\Gamma_2$ be the normalizer of $\Ea$ in $\mathrm{GL}_2(k)$.
The embedding numbers $e_i=e_i(\Ea|\Omega)$ are the following quantities:
$$e_1=\#(X/\Gamma_1),\quad e_2=\#(X/\Gamma_2),
\quad e_3=\#(Y/\Gamma_1),\quad e_4=\#(Y/\Gamma_2).$$
We call them, respectively, first, second, third, and fourth embedding number.
We use the embedding vector 
$\stackrel{\rightarrow}e=(e_1,e_2,e_3,e_4)$ to simplify the statements below.
When $\alpha$ is a real number, we denote its integral part by $[\alpha]
=\mathop{\mathrm{max}}\{n\in\mathbb{Z}|n\leq\alpha\}$.
   For any triple $(r,u,t)\in(\mathbb{Z}_{\geq0})^3$ 
satisfying $v\leq u\leq [r/2]$, for $v=\max\{0,r-t\}$, we consider the cardinality
\begin{equation}\label{999}
\chi(r,u,t)=\#\left\{ \bar a\in\left(\frac{\oink}{\pi^{t-r+2u}\oink}\right)^*\bigg|\begin{array}{c}
\bar a^2=
\bar1 \textnormal{ and } \nu(a-1)=\nu(\pi^{t-r+u})
\\ \textnormal{ for any lifting }a\in\oink\textnormal{ of }\bar{a}\end{array}
\right\},
\end{equation}
 which we set as $1$ for $u=0$. Note that $\nu(a-1)$ depends only on $\bar{a}$ if
$\bar{a}\neq \bar{1}$.
Let $\chi_3=\chi_3(r,t)=\sum_{u=v}^{h}\chi(r,u,t)$, and set it as $0$ if the sum is empty.

\begin{theorem}\label{t22}
Let $L/k$ be a quadratic extension.
Let $\Ea$ be an Eichler order of level $r>0$ and let $\Omega\subseteq\Ea$ be an order isomorphic to 
$\oink_L^{\{t\}}$. Then there exists optimal embeddings of $\Omega$ into $\Ea$ if and only if one
of the following conditions hold:
\begin{itemize}
\item $r\leq 2t$ and $L/k$ is unramified, or
\item $r\leq 2t+1$ and $L/k$ is ramified.
\end{itemize}
Furthermore, when optimal embeddings do exist, the values for the embedding numbers
are given by the formula $$\stackrel{\rightarrow}e=\frac {p^{[r/2]}}2(4,2,2,1)-\frac{m}2(2,1,1,0)
+\frac{\chi_2}4(0,2,0,1)+\frac{\chi_3}4(0,0,2,1),$$
where  $\chi_3$ is as above, while $m$ and $\chi_2$ are as in Table 1.
 If $r=0$, then $\stackrel{\rightarrow}e=(1,1,1,1)$.
\end{theorem}

\begin{table}
\begin{tabular}{|c|c|c|c|}
\hline $L/k$ &
$r$&$m$&$\chi_2$\\
\hline\hline
Unramified&$r=2h+1<2t$&$0$
&$0$\\ \hline
Unramified&$r=2h<2t$&$(q-1)q^{h-1}$
&$\chi(r,h,t)$\\ \hline
Unramified&$r=2t$&$q^{t}$
&$\chi(r,t,t)$\\ \hline
Ramified&$r=2h+1<2t$&$0$
&$0$\\ \hline
Ramified&$r=2h<2t$&$(q-1)q^{h-1}$
&$\chi(r,h,t)$\\ \hline
Ramified&$r=2t$&$(q-1)q^{t-1}$
&$\chi(r,t,t)$\\ \hline
Ramified&$r=2t+1$&$q^t$
&$\chi(r,t,t)$\\ \hline
\hline
\end{tabular}
\caption{The invariants $m$, and $\chi_2$ for the order $\oink_L^{\{t\}}$.}
\end{table}

\section{Trees and Ghost branches}

 The ends of the Bruhat-Tits tree $\mathfrak{t}(k)$, as defined in \S1, are in correspondence with
the elements of the projective line $\mathbb{P}^1(k)$ \cite[\S4]{omeat}. This can be seen by associating,
to each ball $B_a^{[r]}:=B_a\big[|\pi|^r \big]$, the endomorphism ring of the lattice 
$\left\langle\left(\begin{array}{c}a\\1\end{array}\right),\left(\begin{array}{c}\pi^r\\0\end{array}\right)
\right\rangle$, which is a maximal order. This defines a correspondence between balls and maximal orders
that can be used to define a tree whose vertices are balls, while two balls are neighbors if one is a
maximal sub-ball of the other.

 The largest
subgraph whose vertices contain an order $\Omega$, as before, is denoted $\mathfrak{s}(\Omega)$,
or $\mathfrak{s}_k(\Omega)$ is we need to emphasize the field, as it is the case in all that follows. 
If $\Omega=\oink[a_1,\dots,a_n]$, we write $\mathfrak{s}_k(a_1,\dots,a_n)$.
When $\mathfrak{s}_k(\Omega)$ is a thick path, which is usually the case,  as described in \S1,
 we denote its stem by $\mathfrak{m}_k(\Omega)$. 
The notation $\mathfrak{m}_k(a_1,\dots,a_n)$ is defined analogously.
Note that $\mathfrak{s}_k(a_1,\dots,a_n)=\bigcup_{i=1}^n\mathfrak{s}_k(a_i)$, but the corresponding
property for stems is usually false. 

With the preceding definitions, the vertices of Bruhat-Tits tree $\mathfrak{t}(k)$, 
for a finite field extension $L$ of $k$, can be identified with a subset of the vertices in $\mathfrak{t}(L)$,
via $\Da\mapsto\oink_L\otimes_{\oink_k}\Da$.
However,  the map $\mathfrak{t}(k)\hookrightarrow \mathfrak{t}(L)$ is not a morphism of graphs subgraph,
unless $L/k$ is an unramified extension.
To fix this, we normalize, in all that follows, for any finite extension $L/k$, both the valuation in $L$
and the tree distance in $\mathfrak{t}(L)$, in a way that both are extensions of the corresponding
functions defined on $k$. In particular, for any uniformizing parameter $\pi_L\in L$ we set
$\nu(\pi_L)=\frac 1e$, where $e=e(L/k)$ is the ramification index, and let $|a|=c^{\nu(a)}$,
for a suitable positive constant $c<1$, denote the absolute value. Similarly,
 $\delta(v,v')=\frac1e$ for
neighboring vertices in $\mathfrak{t}(L)$. For any pair of ends $a,b\in\mathbb{P}^1(L)$ we denote
by $\mathfrak{p}(a,b)$ the smallest tree containing these two ends. This is a maximal path, i.e.,
a graph whose vertices are precisely the vertices in a double infinite walk. 
For any subgraph $\mathfrak{b}$, and any vertex
$v$, we denote by $p_{k}(v;\mathfrak{b})$ the radius of the largest ball in $\mathfrak{t}(k)$, 
with center $v$, contained in $\mathfrak{b}$.  The depth $p_k(\mathfrak{b})$ of the subgraph 
$\mathfrak{b}$  is the maximal deph of its vertices. Ball radii and depth are normalized according to
our general pattern, e.g., $p_L(v;\mathfrak{b})$ is a multiple of $\frac1e$,
while $B_0^{[1/e]}$ denotes the ball of radius $|\pi_L|$ centered at $0$.
With this conventions, the vertices in $\mathfrak{s}_k(a_1,\dots,a_n)$ can be identified with
the corresponding vertices in  $\mathfrak{s}_L(a_1,\dots,a_n)$. It is not always the case,
however, that the vertices in  $\mathfrak{m}_k(a_1,\dots,a_n)$ are vertices of
$\mathfrak{m}_L(a_1,\dots,a_n)$ (c.f. \S5). 

When $L/k$ is a Galois extension, the galois group
$\mathrm{Gal}(L/k)$ acts on both $\mathbb{P}^1(L)$ and $\mathfrak{t}(L)$. It is not hard to see,
using the explicit correspondence between balls and maximal orders mentioned at the begining of
this section, that these two actions are compatible, 
in the sense that an element $\sigma\in \mathrm{Gal}(L/k)$
maps the maximal path $\mathfrak{p}(a,b)$ defined above onto
the maximal path $\mathfrak{p}\Big(\sigma(a),\sigma(b)\Big)$.

On the other hand, the action of $\mathrm{Gal}(L/k)$ on $V_{\mathfrak{t}(L)}$ leaves invariant many vertices
that fail to belong to $V_{\mathfrak{t}(k)}$. In fact, a ball $B=B_a[|u|]$ is invariant if and only if
$|\sigma(a)-a|\leq |u|$ for every element $\sigma\in \mathrm{Gal}(L/k)$. This condition is satisfied
by every ball of the form $B_a[|u|]$ with $a\in k$ even if $|u|\notin|k^*|$. 
However, there are many invariant balls without an invariant center. An example is the ball
$B_z^{[w]}$ in Figure 3C (c.f. \S5).

There is also a natural action of the group of Moebius transformation on the set of balls that correspond to
the $\mathrm{PSL_2}(k)$-action on maximal orders by conjugation. It can be define by associating, to each ball
$B$, a partition $\mathbb{P}^1(k)=B^{\mathit{c}}\cup B_1\cup\cdots\cup B_q$ of the projective line 
into balls and complements of balls. The balls $B_1,\dots,B_q$ are the neighboring sub-balls of $B$.
Moebius transformations act naturally on those partitions \cite{omeat}.
 This action is compatible with the action of the Galois group. 

In what follows, we refer to vertices in $\mathfrak{t}(L)$ that are not in $\mathfrak{t}(k)$ as ghost vertices.
Similarly, any maximal path $\mathfrak{p}(a,b)\subseteq \mathfrak{t}(L)$ with $a,b\in k$ is identified with
the corresponding path in $\mathfrak{t}(k)$. Maximal paths that are not of this form are called ghost paths.
This is particularly important in the sequel, as the branch of the order, in $\mathbb{M}_2(L)$, generated
by an integral domain $\Omega\subseteq\mathbb{M}_2(k)$,  have a doubly infinite path as a 
stem for an appropiate quadratic extension $L$. We refer to this path as the ghost stem of $\Omega$,
and usually fails to contain the stem $\mathfrak{m}_k(\Omega)$ which has either one or two vertices
\cite{Eichler2}. 
 A similar convention applies to finite paths.

\section{The path of an idempotent}\label{s4}

For every pair $(a,b)$ of different elements in $k$ we define 
$\tau_{a,b}\in\mathbb{M}_2(k)$ as the only idempotent satisfying
the following conditions:
$$\ker(\tau_{a,b})=\left\langle\vvect a1\right\rangle,\qquad \im(\tau_{a,b})=\left\langle\vvect b1\right\rangle.$$
If $a$ or $b$ is $\infty$, we replace the corresponding generator by \scriptsize $\vvect 10$ \normalsize.
Thus $\tau_{a,b}$ is defined for every pair $(a,b)$ of different elements in $\mathbb{P}^1(k)$.
It is easy to see that $\tau_{b,a}=1-\tau_{a,b}$.
Recall that the group $\mathrm{PGL}_2(k)$, which is isomorphic to the group $\mathcal{M}(k)$ of M\"obius 
transformations, acts transitively on both, nontrivial idempotents by conjugation,
 and ordered pairs of distinct elements in 
$\mathbb{P}^1(k)$ as M\"obius  transformations.
  It is immediate from the definitions that both actions are compatible, namely, if $\mu$ is the M\"obius
transformation corresponding to the matrix $A$, then $A\tau_{a,b}A^{-1}=\tau_{\mu(a),\mu(b)}$. 
Since any non-trivial idempotent in $\mathbb{M}_2(k)$ has a one-dimensional image and a one-dimensional 
kernel, every idempotent equals $\tau_{a,b}$ for some ordered pair $(a,b)\in\mathbb{P}^1(k)^2$
 satisfying $a\neq b$. Furthermore, by applying a  M\"obius 
transformation, we can assume $a=0$ and $b=\infty$, so $\tau_{a,b}=\sbmattrix 1000$
is contained in exactly the maximal orders of the form
$\Da_n=\sbmattrix{\oink}{\pi^{-n}\oink}{\pi^n\oink}{\oink}$, which are the ones corresponding to balls
in the path joining $0$ and $\infty$. We conclude that the same holds for every pair 
$(a,b)\in\mathbb{P}^1(k)^2$
where $a\neq b$. Therefore, we can identify idempotents with directed maximal paths,
or more precisely, doubly infinite walks, on the tree. The walk from the end $a$ to the end $b$
contains exactly the maximal orders containing $\tau_{a,b}$.  We use this identification in all that
follows without further ado. 
We can actually give an explicit formula for these idempotents, namely
$$\tau_{a,b}=\frac1{b-a}\bmattrix b{-ab}1{-a}.$$
Many result in this section and \S5 can be proved by extensive computations using the above
formula. We have chosen, however, indirect or geometrical proofs whenever possible for the sake
of brevity.

For any non-trivial idempotent $\tau$, the element $i=1-2\tau$ is a non-trivial solution of $x^2=1$, and conversely,
every non-trivial solution of the equation is of this form. Replacing $\tau$ by $1-\tau$ has the effect of
replacing $i$ by $-i$. Recall that the matrix algebra $\alge=\matrici_2(k)$, together with
the map sending each matrix $A=\sbmattrix xyzw$ to its adjugate matrix $\bar{A}=\sbmattrix w{-y}{-z}x$,
 is isomorphic, as an algebra with involution, to the split quaternion algebra
\scriptsize $\left(\frac{1,-1}k\right)$\normalsize.
 We identify these two algebras in the
remainder of this section.  A matrix satisfying $\bar{A}=-A$ is called a pure
quaternion. By a simple computation, we have $ij+ji=\overline{ij+ji}\in k$
for every pair of pure quaternions $i$ and $j$.

\begin{lemma}\label{rep1}
Let $\lambda\in k$ be any element, and let
$\mathcal{A}(\lambda)$ be the algebra defined, in terms of generators and relations, by
$$\mathcal{A}(\lambda)=k\Big[i,j\Big|i^2=j^2=1, ij+ji=2\lambda\Big].$$
Then there is, up to conjugation,
 a unique representation $\phi:\mathcal{A}(\lambda)\rightarrow \matrici_2(k)$, for which
$\phi(i)$ and $\phi(j)$ are linearly independent. Furthermore, $\phi$ is an isomorphism
unless $\lambda=\pm1$.
\end{lemma}

\begin{proof}
It is immediate from the definition that $\dim_K\Big(\mathcal{A}(\lambda)\Big)\leq4$.
If $\lambda\neq\pm1$, we observe that $i$ and $j'=\lambda i-j$ satisfy the standard relations
among the generators of a quaternion algebra, since
$$ij'=-j'i,\qquad (j')^2=(\lambda i-j)^2=1-\lambda^2\neq0.$$
It follows that $\mathcal{A}(\lambda)$ is a quotient of a quaternion algebra
and therefore it is a quaternion algebra since quaternion algebras are simple. In particular 
$\dim_K\Big(\mathcal{A}(\lambda)\Big)=4$, so that $i$ and $1$ are linearly independent and
 $\mathcal{A}(\lambda)$ cannot be a division algebra since the commutative subalgebra $k[i]$
has too many squares roots of $1$. We conclude that $\mathcal{A}(\lambda)\cong\matrici_2(k)$, so
the result follows from Skolem-Noether's Theorem.

Assume now that $\lambda=1$ or $\lambda=-1$. Replacing $j$ by $-j$ if needed, we can assume that
$\lambda=1$. Let $\phi:\mathcal{A}(1)\rightarrow\matrici_2(k)$ be a representation satisfying the hypoteses.
Let $\omega=\frac{\phi(j)+1}2$ and $\eta=\frac{\phi(i)+1}2$. Note that $\omega$ and $\eta$ are
idempotents, whence we can write $\omega=\tau_{a,b}$ and $\eta=\tau_{c,d}$, and the branches
$S(i)$ and $S(j)$ are tubular neighborhoods of the corresponding paths. Now consider
the nilpotent element $u_n=\pi^{-n}\phi(i-j)$, satisfying 
$u_n\phi(i)+\phi(i)u_n=u_n\phi(j)+\phi(j)u_n=0$. It is not hard to check that $\phi(j)$ and $u_n$
span an order containing also $\phi(i)$. On the other hand, the sequence $\{u_n\}$
leaves every compact subset of  $\matrici_2(k)$ and therefore $i$ and $j$ are contained simultaneously
in infinite many maximal orders. This is only possible if the paths $\mathfrak{p}(a,b)$ and 
$\mathfrak{p}(c,d)$ have a common end. 
They cannot coincide, since this would imply $\eta=1-\omega$ and therefore
$i=-j$, or else $\eta=\omega$ and therefore $i=j$.
 We can assume therefore that $\omega=\tau_{a,b}$ and $\eta=\tau_{a,d}$. Since the group
of M\"obius transformations act transitively on triples in $\mathbb{P}^1(k)^3$, we can assume
$(a,b,d)=(\infty,0,1)$, whence $\omega=\sbmattrix 0001$ and $\eta=\sbmattrix 0101$.
The result follows.
\end{proof}

Now consider a pair $(i,j)$ of pure quaternions  satisfying $i^2=j^2=1$, and let $\lambda=\frac12(ij+ji)$.
By the previous lemma, the conjugacy class of the pair $(i,j)$ is completely
determined by $\lambda$. On the other hand, if $\tau_{a,b}$ and $\tau_{c,d}$ are the 
corresponding idempotents, the orbit of the quartet $(a,b,c,d)$ under M\"obius transformations is
completely determined by the cross-ratio $t=[a,b;c,d]$. In fact, applying
a M\"obius transformation if needed, we can assume
$(a,b,c,d)=(\infty,0,1,t)$. In this case we say that the pair of paths is in the first standard form 
(c.f. Figure 1). It follows that also $t$ is a complete invariant of the quartet.

\begin{figure}
\[ 
\unitlength 1mm 
\linethickness{0.4pt}
\ifx\plotpoint\undefined\newsavebox{\plotpoint}\fi 
\fbox{\begin{picture}(38,18)(0,0)
\multiput(5,3)(.475,.75){11}{{\rule{.4pt}{.4pt}}}
\multiput(9.8,10.6)(.44444,-.75){10}{{\rule{.4pt}{.4pt}}}
\multiput(25,3)(.475,.75){11}{{\rule{.4pt}{.4pt}}}
\multiput(29.8,10.6)(.44444,-.75){10}{{\rule{.4pt}{.4pt}}}
\put(5,1){\makebox(0,0)[cc]{$a$}}
\put(14,1){\makebox(0,0)[cc]{$b$}}
\put(25,1){\makebox(0,0)[cc]{$c$}}
\put(34,1){\makebox(0,0)[cc]{$d$}}
\put(4,17){\makebox(0,0)[cc]{$A$}}
\end{picture}}
\qquad
\fbox{\begin{picture}(32,18)(0,0)
\multiput(5.68,3)(0,1){13}{{\rule{.4pt}{.4pt}}}
\multiput(15,3)(.475,.75){11}{{\rule{.4pt}{.4pt}}}
\multiput(19.8,10.6)(.44444,-.75){10}{{\rule{.4pt}{.4pt}}}
\put(5,1){\makebox(0,0)[cc]{$0$}}
\put(5,17){\makebox(0,0)[cc]{$\infty$}}
\put(15,1){\makebox(0,0)[cc]{$1$}}
\put(24,1){\makebox(0,0)[cc]{$t$}}
\put(24,17){\makebox(0,0)[cc]{$B$}}
\end{picture}}
\qquad
\fbox{\begin{picture}(18.75,18)(0,0)
\multiput(10.18,3)(-.5,.44){8}{{\rule{.4pt}{.4pt}}}
\multiput(6.68,6)(-.42857,-.43){8}{{\rule{.4pt}{.4pt}}}
\multiput(6.93,6)(.44444,.46){10}{{\rule{.4pt}{.4pt}}}
\multiput(10.93,10)(.51786,-.5){15}{{\rule{.4pt}{.4pt}}}
\multiput(11.18,10)(0,.618){8}{{\rule{.4pt}{.4pt}}}
\put(11.25,17){\makebox(0,0)[cc]{$\infty$}}
\put(3.25,1){\makebox(0,0)[cc]{$1$}}
\put(10.75,1){\makebox(0,0)[cc]{$0$}}
\put(18.75,1){\makebox(0,0)[cc]{$t$}}
\put(3,17){\makebox(0,0)[cc]{$C$}}
\end{picture}}
\]
\caption{Two non-intersecting paths (A), its first standard form (B), 
and one of the posible variation with intersecting paths (C).} 
\end{figure}
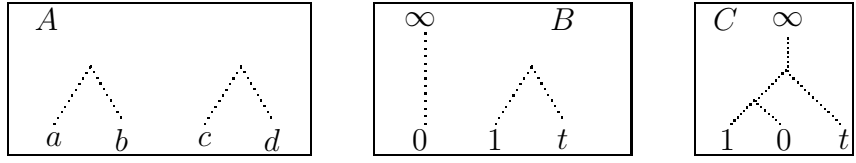

\begin{lemma}
Let $\lambda\in k$ be any element, and let $i$ and $j$ be linearly independent
pure quaternions satisfying $i^2=j^2=1$ and $ij+ji=2\lambda$. Assume $j=2\tau_{a,b}-1$ and
 $i=2\tau_{c,d}-1$. Then 
$$\lambda=\frac{t+1}{t-1},\qquad t=[a,b;c,d].$$
\end{lemma}

\begin{proof}

 Since $\lambda$ is certainly a rational function of  $(a,b,c,d)$, while both $t$ and $\lambda$ are
complete invariants of the quartet,
 it must be of the form $\lambda=\mu(t)$ where $\mu$ is a M\"obius transformation.
 Since transposing $a$ and $b$ has the effect of changing the sign of $i$, and therefore also the sign
of $\lambda$, we have $\mu(t^{-1})=-\mu(t)$.
 We conclude that either $\mu(t)= \xi\cdot\frac{t-1}{t+1}$,
or $\mu(t)= \xi\cdot\frac{t+1}{t-1}$. The first possibility is discarded out since $t=-1$ should 
correspond to a finite  value of $\lambda$. Now $\xi=1$ follows by choosing any particular example. 
For instance, the value $t=\infty$, corresponding to the quartet $(\infty,0,1,\infty)$, gives us
$\tau_{\infty,0}=\sbmattrix 0001$ and $\tau_{\infty,1}=\sbmattrix 0101$, whence 
$$ (2\tau_{\infty,0}-1)(2\tau_{1,\infty}-1)+(2\tau_{1,\infty}-1)(2\tau_{\infty,0}-1)=2\sbmattrix 1001.$$
\end{proof}

Applying the transformation $\mu$ in the preceeding proof to the coordinates of the 
quartet $(\infty,0,1,t)$, we get $(-1,1,\infty,\lambda)$, so that both paths are in one of the 
configurations shown in Figure 2, which we call the second invariant form.
 In the pictures, $|\pi^w|$ is the radius of the smallest ball 
containing $\lambda$ and either element in $\{1,-1\}$. Note that $\nu(\lambda+1)=\varepsilon=:\nu(2)$ in
 Figure 2A, while $w=\nu(\lambda-1)=\nu(\lambda+1)$ in Figure 2C. Next result is now 
apparent:

\begin{figure}
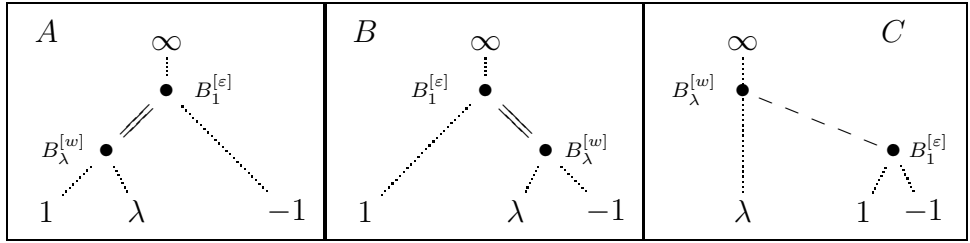

\[ 
\fbox{ \xygraph{
!{<0cm,0cm>;<.8cm,0cm>:<0cm,.8cm>::} 
 !{(0.5,0) }*+{1}="e2" !{(2.5,2.8) }*+{\infty}="e3"
!{(2,0) }*+{\lambda}="e4" !{(4.5,0) }*+{-1}="e5"
  !{(2.5,2) }*+{\bullet}="t3"   !{(1.5,1) }*+{\bullet}="t4" 
!{(3.3,2) }*+{{}^{B_1^{[\varepsilon]}}}="ti3"
!{(0.8,1) }*+{{}^{B_\lambda^{[w]}}}="ti4"
!{(0.5,3) }*+{A}="ti99"
   "e2"-@{.}"t4" "t3"-@{=}"t4" "e3"-@{.}"t3"
"t4"-@{.}"e4" "t3"-@{.}"e5"
 } }
\fbox{ \xygraph{
!{<0cm,0cm>;<.8cm,0cm>:<0cm,.8cm>::} 
 !{(4.5,0) }*+{-1}="e2" !{(2.5,2.8) }*+{\infty}="e3"
!{(3,0) }*+{\lambda}="e4" !{(0.5,0) }*+{1}="e5"
  !{(2.5,2) }*+{\bullet}="t3"   !{(3.5,1) }*+{\bullet}="t4" 
!{(1.6,2) }*+{{}^{B_1^{[\varepsilon]}}}="ti3"
!{(4.2,1) }*+{{}^{B_\lambda^{[w]}}}="ti4"
!{(0.5,3) }*+{B}="ti99"
   "e2"-@{.}"t4" "t3"-@{=}"t4" "e3"-@{.}"t3"
"t4"-@{.}"e4" "t3"-@{.}"e5"
 } }
\fbox{ \xygraph{
!{<0cm,0cm>;<.8cm,0cm>:<0cm,.8cm>::} 
 !{(1,0) }*+{\lambda}="e2" !{(1,2.8) }*+{\infty}="e3"
!{(3,0) }*+{1}="e4" !{(4,0) }*+{-1}="e5"
  !{(3.5,1) }*+{\bullet}="t3"   !{(1,2) }*+{\bullet}="t4" 
!{(4.1,1) }*+{{}^{B_1^{[\varepsilon]}}}="ti3"
!{(0.2,2) }*+{{}^{B_\lambda^{[w]}}}="ti4"
!{(3.5,3) }*+{C}="ti99"
   "e2"-@{.}"e3" "t3"-@{--}"t4"
"t3"-@{.}"e4" "t3"-@{.}"e5"
 } }
\]
\caption{Configuration of paths in the proof of Lemma 4.3.} 
\end{figure}

\begin{lemma}\label{fd1}
Let $\lambda\in K$ be any element, and let $i$ and $j$ be linearly independent
pure quaternions satisfying $i^2=j^2=1$ and $ij+ji=2\lambda$. Let 
$d_f=-\frac12\nu\left(\frac{\lambda^2-1}4\right)$,
 then the following holds:
\begin{enumerate}
\item If $\lambda\neq\pm1$ and $d_f\leq0$, then the stems
$\mathfrak{m}_k(i)$ and $\mathfrak{m}_k(j)$ intersect non-trivially, and the intersection
is a path of length $-2d_f$ (see Figures 2A-B).
\item If $\lambda\neq\pm1$ and $d_f>0$, then $d_f$ is the distance between the stems
$\mathfrak{m}_k(i)$ and $\mathfrak{m}_k(j)$
 (see Figure 2C).  
\item   If $\lambda=\pm1$, the intersection between the two stems is a ray, and $d_f=\infty$.
\end{enumerate}
\end{lemma}

\section{Galois action on ghost branches}

In this section we let $\tau=\tau_{z,u}$ be an idempotent in the algebra $\matrici_2(L)$ generating an
algebra $L[\tau]=L[i]$, where $i\in \matrici_2(k)\subseteq\matrici_2(L)$ is a pure quaternion
satisfying $k[i]\cong L$. Certainly $L/k$ is a quadratic extension. Since the algebra  $k[i]$
has no non-trivial idempotents, necessarily $\tau_{z,u}\notin \matrici_2(k)$, and therefore
$z$ and $u$ cannot be both in $k$. On the other hand, since $L[i]$ is obtained from
$k[i]$ by extension of scalars, and $\tau_{z,u}$ and $\tau_{u,z}$ are the only two idempotents
in this algebra, the non-trivial element $\sigma$ in the Galois group $\mathrm{Gal}(L/k)=\langle\sigma\rangle$
 necessarily permutes these two idempotents,
and therefore also $z$ and $u=\sigma(z)$. If $i^2=\alpha$, we can write $L=k[\sqrt{\alpha}]$, while
$z=a+b\sqrt{\alpha}$ and $u=\bar{z}=a-b\sqrt{\alpha}$. This can be used to easily recover the
description for branches of pure quaternions given in \cite[Lemma 3.4]{osbotbtt}.  

 Let $\xi\in k$
be chosen so that its distance to $z$ is minimal, as in Figure 3A. Then $\xi$ must be equidistant from the
extremes $z$ and $\bar{z}$. In fact, $z$ and $\bar{z}$ are   equidistant from every element in $k$,
so minimizing $|z-\xi|$ is equivalent to minimizing $|(z-\xi)(\bar{z}-\xi)|=|(a-\xi)^2-\alpha b^2|$.
This is achieved by the element 
$\xi=a+b\delta$, where $\delta\in k$ is the element minimizing $|\delta^2-\alpha|$.
In particular, the ideal $(\delta^2-\alpha)$ is the quadratic defect of $\alpha$.
 The path joining $\xi$ and $\infty$ is called a $k$-vine for  $\tau$. The normalized tree-distance from
the path $\mathfrak{s}_L(\tau)$ to the $k$-vine is 
$$\nu(\bar{z}-z)-\nu(\xi-z)=\nu(2b\sqrt\alpha)-\nu\Big(b(\delta-\sqrt\alpha)\Big)$$
$$=
-\nu\left(\frac{\delta-\sqrt\alpha}{2\sqrt\alpha}\right)=
-\frac12\nu\left(\frac{\delta^2-\alpha}{4\alpha}\right).$$

\begin{figure}
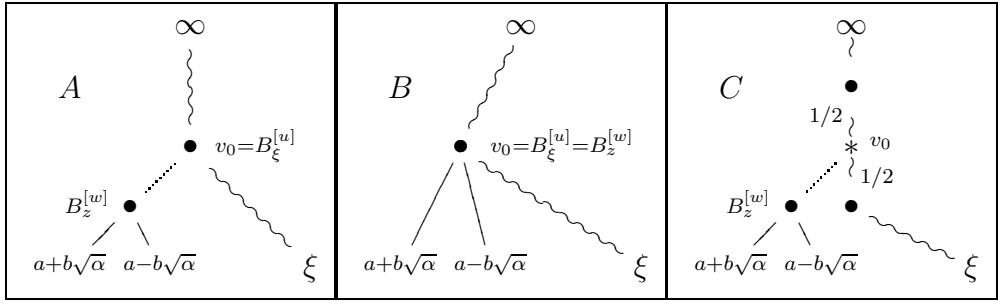

\[ 
\fbox{ \xygraph{
!{<0cm,0cm>;<.8cm,0cm>:<0cm,.8cm>::} 
 !{(0.5,0) }*+{{}^{a+b\sqrt\alpha}}="e2" !{(2.5,4) }*+{\infty}="e3"
!{(2,0) }*+{{}^{a-b\sqrt\alpha}}="e4" !{(4.5,0) }*+{\xi}="e5"
  !{(2.5,2) }*+{\bullet}="t3"   !{(1.5,1) }*+{\bullet}="t4" 
!{(3.6,2) }*+{{}^{v_0=B_\xi^{[u]}}}="ti3"
!{(0.8,1) }*+{{}^{B_z^{[w]}}}="ti4"
!{(0.5,3) }*+{A}="ti99"
   "e2"-"t4" "t3"-@{.}"t4" "e3"-@{~}"t3"
"t4"-"e4" "t3"-@{~}"e5"
 } }
\fbox{ \xygraph{
!{<0cm,0cm>;<.8cm,0cm>:<0cm,.8cm>::} 
 !{(0.5,0) }*+{{}^{a+b\sqrt\alpha}}="e2" !{(2.5,4) }*+{\infty}="e3"
!{(2,0) }*+{{}^{a-b\sqrt\alpha}}="e4" !{(4.5,0) }*+{\xi}="e5"
  !{(1.5,2) }*+{\bullet}="t3"   !{(1.5,2) }*+{\bullet}="t4" 
!{(3.2,2) }*+{{}^{v_0=B_\xi^{[u]}=B_z^{[w]}}}="ti3"
!{(0.5,3) }*+{B}="ti99"
   "e2"-"t4" "t3"-@{.}"t4" "e3"-@{~}"t3"
"t4"-"e4" "t3"-@{~}"e5"
 } }
\fbox{ \xygraph{
!{<0cm,0cm>;<.8cm,0cm>:<0cm,.8cm>::} 
 !{(0.5,0) }*+{{}^{a+b\sqrt\alpha}}="e2" !{(2.5,4) }*+{\infty}="e3"
!{(2,0) }*+{{}^{a-b\sqrt\alpha}}="e4" !{(4.5,0) }*+{\xi}="e5"
  !{(2.5,2) }*+{*}="t3"    !{(2.5,3) }*+{\bullet}="t6"    !{(2.5,1) }*+{\bullet}="t7" 
  !{(1.5,1) }*+{\bullet}="t4" 
!{(3,2) }*+{{}^{v_0}}="ti3"
!{(0.8,1) }*+{{}^{B_z^{[w]}}}="ti4"
!{(0.5,3) }*+{C}="ti99"
   "e2"-"t4" "t3"-@{.}"t4" "e3"-@{~}"t6"
"t4"-"e4" "t7"-@{~}"e5" "t3"-@{~}^{1/2}"t7" "t3"-@{~}^{1/2}"t6"}}
\]
\caption{The $k$-vine of a ghost path. Here $w=v_L(2b\sqrt\alpha)$ and $(b^{-1}\pi^{2u})=
\qdef(\alpha)$ is the quadratic defect,
note that distances are normalized. } 
\end{figure}

Recall that the depth of the branch $\mathfrak{s}_L(i) \subseteq\mathfrak{t}(L)$ is
$p=\nu(2\sqrt\alpha)$ \cite[Lemma 3.4]{osbotbtt}. Note that using the normalized distance
make no difference in this formula.
 We conclude that the vertex in the $k$-vine
that is closest to the stem of the branch $\mathfrak{s}_L(i)$ has depth 
$p=\frac12\nu\left(\delta^2-\alpha\right)$ in that branch. 
Replacing $i$ by $\pi^{-n}i$ if needed, we can always assume that $\alpha\in Q$, as defined in Theorem \ref{t21}.  
\begin{itemize}
\item  If the extension $L/k$ is unramified,  we have that $\alpha=\Delta\in Q$ is a unit of minimal
quadratic defect. The vertex $v_0$ in the $k$-vine
that is closest to the stem $\mathfrak{m}_L(i)$ has depth 
$p=-\nu(2)$. We conclude that $v_0$ in Figure 3B
is indeed a vertex of $\mathfrak{t}(k)$, and therefore it is the stem of $\mathfrak{s}_k(i)$. 
The depth  of this branch, in this case, equals $p$.
 \item If the extension $L/k$ is ramified,  the vertex $v_0$ in the $k$-vine
that is closest to the stem $\mathfrak{m}_L(i)$ has depth 
$p=-\nu\left(\delta^2-\alpha\right)$ in that branch, which is an odd number,
as it is the valuation of the quadratic defect \cite[63.2]{Om}. We conclude that $v_0$
is a ghost vertex, and therefore the midpoint of the stem $\mathfrak{m}_k(i)$ (see Figure 3C).
As the normalized distance in $\mathfrak{t}(L)$ of neighboring vertices is $1/2$,
the depth  of this branch is, therefore, $\frac{p-1}2$.
\end{itemize}

Now assume we have two pure quaternions $i$ and $j$ en $\matrici_2(k)$ satisfying the relations
\begin{equation}\label{alfabeta}
i^2=\alpha,\qquad j^2=\beta,\qquad ij+ji=2\lambda.
\end{equation}
Then, the elements $i_0=\frac i{\sqrt\alpha}$, and $j_0=\frac j{\sqrt{\beta}}$ in $\matrici_2(L)$,
where $L=K(\sqrt\alpha,\sqrt\beta)$, satisfy the relations
$$i_0^2=j_0^2=1,\qquad i_0j_0+j_0i_0=2\lambda'=\frac{2\lambda}{\sqrt{\alpha\beta}},$$
so that the results in last section apply to them.
In this general setting, we have a classification of pairs of pure quaternions satisfying the relations
in (\ref{alfabeta}) that is entirely analog to the one in \S\ref{s4}, namely:

\begin{lemma}\label{rep2}
Let $\alpha,\beta,\lambda\in k$, with $\alpha\beta\neq0$, and let
$\mathcal{A}=\mathcal{A}(\alpha,\beta,\lambda)$ be the algebra defined, in terms of generators and relations, by
$$\mathcal{A}=k\Big[i,j\Big|i^2=\alpha, j^2=\beta, ij+ji=2\lambda\Big].$$
Then there is, up to conjugation, at most one
representation $\phi:\mathcal{A}\rightarrow \matrici_2(k)$, for which
$\phi(i)$ and $\phi(j)$ are linearly independent. Furthermore $\phi$, if it exists, is an isomorphism
unless $\lambda^2=\pm\alpha\beta$. In the latter case $\alpha$ and $\beta$ are squares.
\end{lemma}

\begin{proof}
If $\lambda\neq\pm\sqrt{\alpha\beta}$, it follows from Lemma \ref{rep1} that
 $L\otimes_k\mathcal{A}$, for $L$ as above, 
is a quaternion algebra. We conclude that also $\mathcal{A}$
is a quaternion algebra and the result follows from Skolem-Noether Theorem as
before. We assume thus $\lambda^2=\alpha\beta$, so $\beta/\alpha=(\lambda/\alpha)^2$.
 Replacing $j$ by $\frac {\alpha j}\lambda$ if needed, we can assume $\beta=\alpha=\pm\lambda$.
If this common value is a square, we are in the case of Lemma \ref{rep1}, so we assume this is not the case.
Then any $2$-dimensional representation of $\mathcal{A}$ restrict to a two dimensional
representation of $k[i]$ which is unique up to conjugation, so we may assume $\phi(i)=\sbmattrix 0\alpha10$.
Now we claim that the only matrices $X$ satisfying $X\phi(i)+\phi(i)X=\pm 2\alpha$ and 
$X^2=\alpha$ are $X=\pm\phi(i)$. In fact, the first condition implies $X=\sbmattrix x{\alpha(\pm2-z)}z{-x}$,
while the second gives $x^2=\alpha(-z\pm1)^2$, which implies $x=-z\pm1=0$, since $\alpha$ is not a square.
 The result follows.
\end{proof}

Now we characterize exactly when, given non-zero elements $\alpha,\beta\in k$, there exists two linearly
independent pure quaternions in $\matrici_2(k)$ satisfying the relations in (\ref{alfabeta}).  
Note that we can assume $\lambda^2\neq\alpha\beta$ by Lemma \ref{rep2}, since otherwise,
replacing $i$ and $j$ by suitable multiples if needed, we can assume that  $\alpha=\beta=1$, and in this case
the representation exists by  Lemma \ref{rep1}. In case $\lambda^2\neq\alpha\beta$, quaternions
satisfying   (\ref{alfabeta}) exist precisely when $\mathcal{A}(\alpha,\beta,\lambda)\cong
\mathbb{M}_2(k)$. When either $\alpha$ or $\beta$ is a square,
the algebra $\mathcal{A}(\alpha,\beta,\lambda)$ contains a non-trivial idempotent,
either $\frac{i+\sqrt{\alpha}}{2\sqrt{\alpha}}$ or $\frac{j+\sqrt{\beta}}{2\sqrt{\beta}}$,
and therefore it is a matrix algebra. In the remaining case we apply next result:

\begin{lemma}\label{rep3}
Let $\alpha,\beta,\lambda\in k$, with $\lambda^2\neq\alpha\beta\neq0$.
Assume neither $\alpha$ nor $\beta$ is a square. Then $\mathcal{A}(\alpha,\beta,\lambda)$, defined as
above, is a matrix algebra, if and only if there exists elements $a,b,c,d\in k$ with $bd\neq0$,
 satisfying the relation 
\begin{equation}\label{e3}
\lambda=\frac{b^2\alpha+d^2\beta-(a-c)^2}{2bd}.
\end{equation}
\end{lemma}

\begin{proof}
We can assume that $i=\sqrt\alpha(2\tau_{z_1,z_2}-1)$ and 
$j=\sqrt\beta(2\tau_{z_3,z_4}-1)$, where
$$(z_1,z_2,z_3,z_4)=
(a+b\sqrt\alpha,a-b\sqrt\alpha,c+d\sqrt\beta,c-d\sqrt\beta).$$
Then, the formulas in the preceding section
give 
\begin{equation}\label{eq2}
\frac{2\lambda}{\sqrt{\alpha\beta}}=
\frac{i}{\sqrt\alpha}\frac{j}{\sqrt\beta}+\frac{j}{\sqrt\beta}\frac{i}{\sqrt\alpha}=
2\frac{t+1}{t-1},\end{equation}
where $t=[z_1,z_2;z_3,z_4]$,
which under a little algebraic manipulation becomes
$$\frac{\lambda}{\sqrt{\alpha\beta}}
=\frac{(z_1-z_4)(z_2-z_3)+(z_1-z_3)(z_2-z_4)}{(z_1-z_4)(z_2-z_3)-(z_1-z_3)(z_2-z_4)}=$$
$$-\frac{2(z_1z_2+z_3z_4)-(z_1+z_2)(z_3+z_4)}{(z_1-z_2)(z_3-z_4)}=
-\frac{2(a^2-b^2\alpha)+2(c^2-d^2\beta)-4ac}{4bd\sqrt{\alpha\beta}}.$$
Conversely, if elements $a,b,c,d$ satisfying (\ref{e3}) exist, the above formulas define two paths that
are invariant under the Galois group, so they actually correspond to matrices $i$ and $j$ with coeficients
in $k$.
\end{proof}

Note that, when $\lambda=0$, this reduces to the well known criterion stating that
the quaternion algebra $\left(\frac{\alpha,\beta}k\right)$ splits if and only if
 the quadratic form in the numerator is isotropic.

\begin{lemma}\label{rep3}
Let $L/k$ be a multiple quadratic extension of non-archimedean local fields with Galois group 
$G=\mathrm{Gal}(L/k)$, and let $(x_1,x_2,x_3),
(x'_1,x'_2,x'_3)\in L^3$ be two triples satisfying the following conditions
\begin{enumerate}
\item Each set $\{x_1,x_2,x_3\}$ and $\{x'_1,x'_2,x'_3\}$ is $G$ invariant.
\item For any $\sigma\in G$ and any $i,j\in\{1,2,3\}$ we have $\sigma(x_i)=x_j$ if and only if
$\sigma(x'_i)=x'_j$.
\end{enumerate}  
Then the M\"obius transformation $\tau$ satisfying $\tau(x_i)=x'_i$ has coeficients in $k$. 
\end{lemma}

\begin{proof}
It follows from the hypotheses, and the uniqueness of the M\"obius transformation taking one
ordered triple onto another, that for any matrix $A\in\mathrm{GL}_2(L)$ defining $\tau$, and for any 
element $\sigma\in G$, we have $\sigma(A)=\lambda_\sigma A$, where the map
$\sigma\mapsto\lambda_\sigma$ is a cocycle with values in $L^*$, so the result follows from 
Hilbert's Theorem 90.
\end{proof}

For instance, if $a,b,c,d\in k$, the Moebius transformation sending the triple $(c+d\sqrt{\alpha},c-d\sqrt{\alpha},a)$
onto $(c+d\sqrt{\alpha},c-d\sqrt{\alpha},b)$ is defined over $k$.

\section{Proof of Theorem \ref{t21}}

If $\alpha=\beta=1$, then  the result is a direct aplication of Lemma \ref{fd1}. If one of them, say $\alpha$
equals $\Delta$, the stem $\mathfrak{m}_k(i)$ of the branch $\mathfrak{s}_k(i)$ contains exactly one point,
 namely the highest point $v_0$ of the ghost path $\mathfrak{m}_L(i)$. 
It is clear from Figure 4A, since $v_0$ is the only point in the ghost path
that is defined over $k$, that the distance from any vertex $u\in V_{\mathfrak{t}(k)}$
to $\mathfrak{m}_L(i)$ equals the distance from $u$ to $v_0$, and therefore
the same formulas hold  in this case, after a substitution as in (\ref{eq2}). Note that the fake distance $d_f$
is nonnegative here.
 If $\alpha=\beta=\Delta$, the same argument holds (Figure 4B), unless both ghost paths do
intersect nontrivially, which means that the fake distance is negative. In that case the only
vertex defined over $k$ on that intersection is, necessarily, the highest point on each ghost path
(Figure 4C). The intersection consists, therefore, of a single vertex, as required in this case, since
$l(i)=l(j)=0$.

\begin{figure}
\[
\unitlength 1mm 
\linethickness{0.4pt}
\ifx\plotpoint\undefined\newsavebox{\plotpoint}\fi 
\begin{picture}(118,31.5)(0,0)
\put(25,6){\line(0,1){13}}
\put(25,20.5){\line(0,1){6}}
\put(25,19.5){\makebox(0,0)[cc]{$\bullet$}}
\put(25,30){\makebox(0,0)[cc]{$\infty$}}
\put(14,13.25){\makebox(0,0)[cc]{$\bullet$}}
\put(11,15){\makebox(0,0)[cc]{${}^{v_0}$}}
\multiput(13.18,13.68)(-.0333333,-.0516667){15}{\line(0,-1){.0516667}}
\multiput(12.18,12.13)(-.0333333,-.0516667){15}{\line(0,-1){.0516667}}
\multiput(11.18,10.58)(-.0333333,-.0516667){15}{\line(0,-1){.0516667}}
\multiput(10.18,9.03)(-.0333333,-.0516667){15}{\line(0,-1){.0516667}}
\multiput(9.18,7.48)(-.0333333,-.0516667){15}{\line(0,-1){.0516667}}
\multiput(13.93,13.43)(.0316667,-.0483333){15}{\line(0,-1){.0483333}}
\multiput(14.88,11.98)(.0316667,-.0483333){15}{\line(0,-1){.0483333}}
\multiput(15.83,10.53)(.0316667,-.0483333){15}{\line(0,-1){.0483333}}
\multiput(16.78,9.08)(.0316667,-.0483333){15}{\line(0,-1){.0483333}}
\multiput(17.73,7.63)(.0316667,-.0483333){15}{\line(0,-1){.0483333}}
\multiput(13.93,14.43)(.82692,.42308){14}{{\rule{.4pt}{.4pt}}}
\put(14.25,24){\makebox(0,0)[cc]{A}}
\put(48,14.5){\makebox(0,0)[cc]{$\bullet$}}
\multiput(47.18,14.93)(-.0333333,-.0516667){15}{\line(0,-1){.0516667}}
\multiput(46.18,13.38)(-.0333333,-.0516667){15}{\line(0,-1){.0516667}}
\multiput(45.18,11.83)(-.0333333,-.0516667){15}{\line(0,-1){.0516667}}
\multiput(44.18,10.28)(-.0333333,-.0516667){15}{\line(0,-1){.0516667}}
\multiput(43.18,8.73)(-.0333333,-.0516667){15}{\line(0,-1){.0516667}}
\multiput(47.93,14.68)(.0316667,-.0483333){15}{\line(0,-1){.0483333}}
\multiput(48.88,13.23)(.0316667,-.0483333){15}{\line(0,-1){.0483333}}
\multiput(49.83,11.78)(.0316667,-.0483333){15}{\line(0,-1){.0483333}}
\multiput(50.78,10.33)(.0316667,-.0483333){15}{\line(0,-1){.0483333}}
\multiput(51.73,8.88)(.0316667,-.0483333){15}{\line(0,-1){.0483333}}
\put(69.75,15.25){\makebox(0,0)[cc]{$\bullet$}}
\multiput(68.93,15.68)(-.0333333,-.0516667){15}{\line(0,-1){.0516667}}
\multiput(67.93,14.13)(-.0333333,-.0516667){15}{\line(0,-1){.0516667}}
\multiput(66.93,12.58)(-.0333333,-.0516667){15}{\line(0,-1){.0516667}}
\multiput(65.93,11.03)(-.0333333,-.0516667){15}{\line(0,-1){.0516667}}
\multiput(64.93,9.48)(-.0333333,-.0516667){15}{\line(0,-1){.0516667}}
\multiput(69.68,15.43)(.0316667,-.0483333){15}{\line(0,-1){.0483333}}
\multiput(70.63,13.98)(.0316667,-.0483333){15}{\line(0,-1){.0483333}}
\multiput(71.58,12.53)(.0316667,-.0483333){15}{\line(0,-1){.0483333}}
\multiput(72.53,11.08)(.0316667,-.0483333){15}{\line(0,-1){.0483333}}
\multiput(73.48,9.63)(.0316667,-.0483333){15}{\line(0,-1){.0483333}}
\multiput(58.18,28.18)(-.55556,-.68056){19}{{\rule{.4pt}{.4pt}}}
\multiput(58.93,27.93)(.58333,-.66667){19}{{\rule{.4pt}{.4pt}}}
\put(58.5,28.75){\makebox(0,0)[cc]{$\bullet$}}
\put(49.5,28){\makebox(0,0)[cc]{B}}
\put(91.5,15.25){\makebox(0,0)[cc]{$\bullet$}}
\multiput(90.68,15.68)(-.0333333,-.0516667){15}{\line(0,-1){.0516667}}
\multiput(89.68,14.13)(-.0333333,-.0516667){15}{\line(0,-1){.0516667}}
\multiput(88.68,12.58)(-.0333333,-.0516667){15}{\line(0,-1){.0516667}}
\multiput(87.68,11.03)(-.0333333,-.0516667){15}{\line(0,-1){.0516667}}
\multiput(86.68,9.48)(-.0333333,-.0516667){15}{\line(0,-1){.0516667}}
\multiput(91.43,15.43)(.0316667,-.0483333){15}{\line(0,-1){.0483333}}
\multiput(92.38,13.98)(.0316667,-.0483333){15}{\line(0,-1){.0483333}}
\multiput(93.33,12.53)(.0316667,-.0483333){15}{\line(0,-1){.0483333}}
\multiput(94.28,11.08)(.0316667,-.0483333){15}{\line(0,-1){.0483333}}
\multiput(95.23,9.63)(.0316667,-.0483333){15}{\line(0,-1){.0483333}}
\put(113.25,16){\makebox(0,0)[cc]{$\bullet$}}
\multiput(112.43,16.43)(-.0333333,-.0516667){15}{\line(0,-1){.0516667}}
\multiput(111.43,14.88)(-.0333333,-.0516667){15}{\line(0,-1){.0516667}}
\multiput(110.43,13.33)(-.0333333,-.0516667){15}{\line(0,-1){.0516667}}
\multiput(109.43,11.78)(-.0333333,-.0516667){15}{\line(0,-1){.0516667}}
\multiput(108.43,10.23)(-.0333333,-.0516667){15}{\line(0,-1){.0516667}}
\multiput(113.18,16.18)(.0316667,-.0483333){15}{\line(0,-1){.0483333}}
\multiput(114.13,14.73)(.0316667,-.0483333){15}{\line(0,-1){.0483333}}
\multiput(115.08,13.28)(.0316667,-.0483333){15}{\line(0,-1){.0483333}}
\multiput(116.03,11.83)(.0316667,-.0483333){15}{\line(0,-1){.0483333}}
\multiput(116.98,10.38)(.0316667,-.0483333){15}{\line(0,-1){.0483333}}
\put(102,29.5){\makebox(0,0)[cc]{$\bullet$}}
\put(93,28.75){\makebox(0,0)[cc]{C}}
\multiput(91.18,15.93)(.03333333,.0452381){45}{\line(0,1){.0452381}}
\multiput(94.18,20.001)(.03333333,.0452381){45}{\line(0,1){.0452381}}
\multiput(97.18,24.073)(.03333333,.0452381){45}{\line(0,1){.0452381}}
\multiput(100.18,28.144)(.03333333,.0452381){45}{\line(0,1){.0452381}}
\multiput(101.68,30.18)(.03338509,-.04037267){46}{\line(0,-1){.04037267}}
\multiput(104.751,26.465)(.03338509,-.04037267){46}{\line(0,-1){.04037267}}
\multiput(107.823,22.751)(.03338509,-.04037267){46}{\line(0,-1){.04037267}}
\multiput(110.894,19.037)(.03338509,-.04037267){46}{\line(0,-1){.04037267}}
\put(9,4){\makebox(0,0)[cc]{${}^z$}}
\put(18,4){\makebox(0,0)[cc]{${}^{\bar{z}}$}}
\put(43,5){\makebox(0,0)[cc]{${}^z$}}
\put(52,5){\makebox(0,0)[cc]{${}^{\bar{z}}$}}
\put(65,5){\makebox(0,0)[cc]{${}^u$}}
\put(74,5){\makebox(0,0)[cc]{${}^{\bar{u}}$}}
\put(87,6){\makebox(0,0)[cc]{${}^z$}}
\put(96,6){\makebox(0,0)[cc]{${}^u$}}
\put(108,6){\makebox(0,0)[cc]{${}^{\bar{z}}$}}
\put(118,6){\makebox(0,0)[cc]{${}^{\bar{u}}$}}
\end{picture}
\]
\caption{Relative position for $\alpha=\Delta$ and $\beta=1$ (A) and $\beta=\Delta$ (B,C).
The long dashes in (C) denote the ghost intersection 
$\mathfrak{m}_L(i)\cap\mathfrak{m}_L(j)$ of the two ghost stems.}
\end{figure}
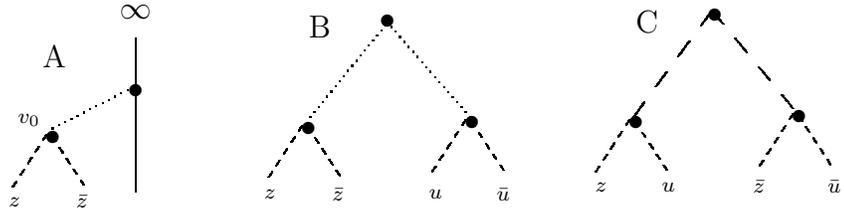

Assume  now that $\alpha=1$, while $\beta$ is either a ramified unit of a uniformizing parameter.
Then Lemma \ref{fd1}, and a substitution as in (\ref{eq2}),
prove that the normalized distance, between the stems $\mathfrak{m}_L(i)$ and $\mathfrak{m}_L(j)$  is
$-\frac{1}{2}\nu \left( \frac{\lambda^2-\beta}{4\beta} \right)$, which is always possitive by the properties of
the quadratic deffect. However,
the stem $\mathfrak{m}_k(j)$  lies at a distance $\nu(2\sqrt\beta)-\nu(\eta-\sqrt\beta)=
-\frac12\nu\left(\frac{\eta^2-\beta}{4\beta}\right)$ from the stem
$\mathfrak{m}_L(j)$, assuming that the base $\xi$ of the $k$-vine in Figure 3C satisfies $\xi=a+b\eta$,
where $\qdef(\beta)=(\eta^2-\beta)$.
 It follows that, if the stem $\mathfrak{m}_L(i)$, which is identified as
a path with the stem $\mathfrak{m}_k(i)$, fails to contain the stem $\mathfrak{m}_k(j)$, then
the distance between these stems is \small
$$-\frac{1}{2}\nu\left( \frac{\lambda^2-\beta}{4\beta} \right)+
\frac12\nu\left(\frac{\eta^2-\beta}{4\beta}\right)-\frac12=-
\frac12\nu\left(\frac{\lambda^2-\beta}{\eta^2-\beta}\right)-\frac12=t-\frac12\nu(\lambda^2-\beta)=d_f,$$
\normalsize
where the summand $1/2$ is the normalized distance between $v_0$ and an endpoint of
the stem $\mathfrak{m}_k(j)$, as in Figure 5B.
Note that  $\mathfrak{m}_k(i)$ contains $\mathfrak{m}_k(j)$ precisely when the fake distance is $-1/2$,
as this is the case if and only if the $k$-vine can be chosen in a way that $\eta=\lambda$ as in Figure 5A.
Furthermore, $d_f=-1/2$ is indeed the minimum possible value since the valuation 
\scriptsize $\nu\!\left(\frac{\lambda^2-\beta}{\eta^2-\beta}\right)$ \normalsize 
above cannot be positive by definition of $\eta$.
The case where $\alpha=\Delta$,
and $\beta$ is either a ramified unit of a uniformizing parameter, is handled similarly, see Figure 6.
Cases B and C can easily be reduced to case A, by applying a Moebius transformation
taking $\infty$ to $\eta$, while preserving the ends of $\mathfrak{m}_L(i)$ (c.f. Lemma \ref{rep3}).
In this case $d_f=0$ is possible, when the vertex in $\mathfrak{m}_k(j)$, 
the highest point of $\mathfrak{m}_L(j)$, coincide with the lower
endpoint of $\mathfrak{m}_k(i)$ in Figure 6A. However, $d_f=-\frac12$ is not possible,
as  the vertex in $\mathfrak{m}_k(j)$ is defined over $k$, and cannot, therefore,
be the midpoint of $\mathfrak{m}_k(i)$. Note that the stems $\mathfrak{m}_L(i)$ and 
$\mathfrak{m}_L(j)$ are defined over  different quadratic extensions of $k$ in this case.

\begin{figure}
\[
\unitlength 1mm 
\linethickness{0.4pt}
\ifx\plotpoint\undefined\newsavebox{\plotpoint}\fi 
\begin{picture}(64.25,25.5)(0,0)
\multiput(17.25,25)(.0336787565,-.0466321244){386}{\line(0,-1){.0466321244}}
\put(21,20){\makebox(0,0)[cc]{$\bullet$}}
\put(23.75,16.25){\makebox(0,0)[cc]{$\bullet$}}
\multiput(21.68,18.68)(-.61111,-.61111){10}{{\rule{.4pt}{.4pt}}}
\multiput(15.68,12.68)(-.0330882,-.0367647){17}{\line(0,-1){.0367647}}
\multiput(14.555,11.43)(-.0330882,-.0367647){17}{\line(0,-1){.0367647}}
\multiput(13.43,10.18)(-.0330882,-.0367647){17}{\line(0,-1){.0367647}}
\multiput(12.305,8.93)(-.0330882,-.0367647){17}{\line(0,-1){.0367647}}
\multiput(16.43,12.93)(.0333333,-.047619){15}{\line(0,-1){.047619}}
\multiput(17.43,11.501)(.0333333,-.047619){15}{\line(0,-1){.047619}}
\multiput(18.43,10.073)(.0333333,-.047619){15}{\line(0,-1){.047619}}
\multiput(19.43,8.644)(.0333333,-.047619){15}{\line(0,-1){.047619}}
\multiput(51.75,25.5)(.0336787565,-.0466321244){386}{\line(0,-1){.0466321244}}
\multiput(56.18,18.68)(-.60714,-.46429){8}{{\rule{.4pt}{.4pt}}}
\multiput(46.18,12.18)(-.60714,-.46429){8}{{\rule{.4pt}{.4pt}}}
\put(51.75,15.25){\makebox(0,0)[cc]{$\bullet$}}
\put(49.9,14.25){\makebox(0,0)[cc]{$+$}}
\put(47.9,15.9){\makebox(0,0)[cc]{${}_{v_0}$}}
\put(47.5,12.5){\makebox(0,0)[cc]{$\bullet$}}
\multiput(51.75,15.55)(-.04878049,-.03353659){82}{\line(-1,0){.04878049}}
\multiput(49.55,14.43)(.5,-.6){6}{{\rule{.4pt}{.4pt}}}
\multiput(52.68,11.18)(.0324074,-.0347222){18}{\line(0,-1){.0347222}}
\multiput(53.846,9.93)(.0324074,-.0347222){18}{\line(0,-1){.0347222}}
\multiput(55.013,8.68)(.0324074,-.0347222){18}{\line(0,-1){.0347222}}
\multiput(52.68,11.43)(-.0318627,-.0343137){17}{\line(0,-1){.0343137}}
\multiput(51.596,10.263)(-.0318627,-.0343137){17}{\line(0,-1){.0343137}}
\multiput(50.513,9.096)(-.0318627,-.0343137){17}{\line(0,-1){.0343137}}
\put(14.5,21){\makebox(0,0)[cc]{A}}
\put(46.25,23.75){\makebox(0,0)[cc]{B}}
\put(30,4){\makebox(0,0)[cc]{${}^{\eta=\lambda}$}}
\put(41,5){\makebox(0,0)[cc]{${}^{\lambda}$}}
\put(66,5){\makebox(0,0)[cc]{${}^{\eta}$}}
\end{picture}
\]
\caption{Relative position for $\alpha=1$ and $\beta$ a ramified unit or a uniformizing parameter.
The stem $\mathfrak{m}_k(j)$ can be contained in the stem $\mathfrak{m}_k(i)$  (A) or not (B).}
\end{figure}
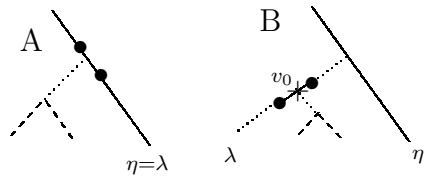

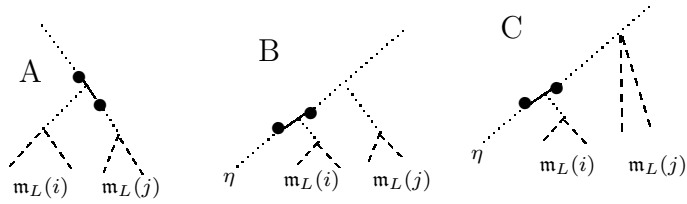
\begin{figure}
\[
\unitlength 1mm 
\linethickness{0.4pt}
\ifx\plotpoint\undefined\newsavebox{\plotpoint}\fi 
\begin{picture}(97.25,29.5)(0,0)
\put(21,20){\makebox(0,0)[cc]{$\bullet$}}
\put(23.75,16.25){\makebox(0,0)[cc]{$\bullet$}}
\multiput(21.68,18.68)(-.61111,-.61111){10}{{\rule{.4pt}{.4pt}}}
\multiput(15.68,12.68)(-.0330882,-.0367647){17}{\line(0,-1){.0367647}}
\multiput(14.555,11.43)(-.0330882,-.0367647){17}{\line(0,-1){.0367647}}
\multiput(13.43,10.18)(-.0330882,-.0367647){17}{\line(0,-1){.0367647}}
\multiput(12.305,8.93)(-.0330882,-.0367647){17}{\line(0,-1){.0367647}}
\multiput(16.43,12.93)(.0333333,-.047619){15}{\line(0,-1){.047619}}
\multiput(17.43,11.501)(.0333333,-.047619){15}{\line(0,-1){.047619}}
\multiput(18.43,10.073)(.0333333,-.047619){15}{\line(0,-1){.047619}}
\multiput(19.43,8.644)(.0333333,-.047619){15}{\line(0,-1){.047619}}
\put(51.75,15.25){\makebox(0,0)[cc]{$\bullet$}}
\put(47.5,13.25){\makebox(0,0)[cc]{$\bullet$}}
\multiput(51.75,15.75)(-.04878049,-.03353659){82}{\line(-1,0){.04878049}}
\multiput(50.18,14.43)(.5,-.6){6}{{\rule{.4pt}{.4pt}}}
\multiput(52.68,11.18)(.0324074,-.0347222){18}{\line(0,-1){.0347222}}
\multiput(53.846,9.93)(.0324074,-.0347222){18}{\line(0,-1){.0347222}}
\multiput(55.013,8.68)(.0324074,-.0347222){18}{\line(0,-1){.0347222}}
\multiput(52.68,11.43)(-.0318627,-.0343137){17}{\line(0,-1){.0343137}}
\multiput(51.596,10.263)(-.0318627,-.0343137){17}{\line(0,-1){.0343137}}
\multiput(50.513,9.096)(-.0318627,-.0343137){17}{\line(0,-1){.0343137}}
\put(14.5,21){\makebox(0,0)[cc]{A}}
\put(46.25,23.75){\makebox(0,0)[cc]{B}}
\multiput(21.5,19.75)(.03333333,-.05){50}{\line(0,-1){.05}}
\multiput(20.18,21.18)(-.53125,.71875){9}{{\rule{.4pt}{.4pt}}}
\multiput(23.68,16.43)(.45,-.7){6}{{\rule{.4pt}{.4pt}}}
\multiput(26.18,12.43)(-.031746,-.079365){9}{\line(0,-1){.079365}}
\multiput(25.608,11.001)(-.031746,-.079365){9}{\line(0,-1){.079365}}
\multiput(25.037,9.573)(-.031746,-.079365){9}{\line(0,-1){.079365}}
\multiput(24.465,8.144)(-.031746,-.079365){9}{\line(0,-1){.079365}}
\multiput(26.43,11.93)(.032967,-.0521978){13}{\line(0,-1){.0521978}}
\multiput(27.287,10.573)(.032967,-.0521978){13}{\line(0,-1){.0521978}}
\multiput(28.144,9.215)(.032967,-.0521978){13}{\line(0,-1){.0521978}}
\multiput(29.001,7.858)(.032967,-.0521978){13}{\line(0,-1){.0521978}}
\multiput(56.68,18.43)(.53125,-.71875){9}{{\rule{.4pt}{.4pt}}}
\multiput(60.93,12.68)(-.033333,-.083333){9}{\line(0,-1){.083333}}
\multiput(60.33,11.18)(-.033333,-.083333){9}{\line(0,-1){.083333}}
\multiput(59.73,9.68)(-.033333,-.083333){9}{\line(0,-1){.083333}}
\multiput(60.93,12.43)(.0324074,-.0324074){18}{\line(1,0){.0324074}}
\multiput(62.096,11.263)(.0324074,-.0324074){18}{\line(1,0){.0324074}}
\multiput(63.263,10.096)(.0324074,-.0324074){18}{\line(0,-1){.0324074}}
\multiput(52.18,15.93)(.69118,.60294){18}{{\rule{.4pt}{.4pt}}}
\multiput(47.43,12.93)(-.6875,-.59375){9}{{\rule{.4pt}{.4pt}}}
\put(84.5,18.5){\makebox(0,0)[cc]{$\bullet$}}
\put(80.25,16.5){\makebox(0,0)[cc]{$\bullet$}}
\multiput(84.5,19)(-.04878049,-.03353659){82}{\line(-1,0){.04878049}}
\multiput(82.93,17.68)(.5,-.6){6}{{\rule{.4pt}{.4pt}}}
\multiput(85.43,14.43)(.0324074,-.0347222){18}{\line(0,-1){.0347222}}
\multiput(86.596,13.18)(.0324074,-.0347222){18}{\line(0,-1){.0347222}}
\multiput(87.763,11.93)(.0324074,-.0347222){18}{\line(0,-1){.0347222}}
\multiput(85.43,14.68)(-.0318627,-.0343137){17}{\line(0,-1){.0343137}}
\multiput(84.346,13.513)(-.0318627,-.0343137){17}{\line(0,-1){.0343137}}
\multiput(83.263,12.346)(-.0318627,-.0343137){17}{\line(0,-1){.0343137}}
\multiput(84.93,19.18)(.69118,.60294){18}{{\rule{.4pt}{.4pt}}}
\multiput(80.18,16.18)(-.6875,-.59375){9}{{\rule{.4pt}{.4pt}}}
\put(78.5,26.75){\makebox(0,0)[cc]{C}}
\put(92.93,25.68){\line(0,-1){.9808}}
\put(92.968,23.718){\line(0,-1){.9808}}
\put(93.007,21.757){\line(0,-1){.9808}}
\put(93.045,19.795){\line(0,-1){.9808}}
\put(93.084,17.834){\line(0,-1){.9808}}
\put(93.122,15.872){\line(0,-1){.9808}}
\put(93.16,13.91){\line(0,-1){.9808}}
\multiput(93.18,25.43)(.031746,-.10119){9}{\line(0,-1){.10119}}
\multiput(93.751,23.608)(.031746,-.10119){9}{\line(0,-1){.10119}}
\multiput(94.323,21.787)(.031746,-.10119){9}{\line(0,-1){.10119}}
\multiput(94.894,19.965)(.031746,-.10119){9}{\line(0,-1){.10119}}
\multiput(95.465,18.144)(.031746,-.10119){9}{\line(0,-1){.10119}}
\multiput(96.037,16.323)(.031746,-.10119){9}{\line(0,-1){.10119}}
\multiput(96.608,14.501)(.031746,-.10119){9}{\line(0,-1){.10119}}
\put(16,5){\makebox(0,0)[cc]{${}^{\mathfrak{m}_L(i)}$}}
\put(28,5){\makebox(0,0)[cc]{${}^{\mathfrak{m}_L(j)}$}}
\put(41,6){\makebox(0,0)[cc]{${}^{\eta}$}}
\put(52,6){\makebox(0,0)[cc]{${}^{\mathfrak{m}_L(i)}$}}
\put(64,6){\makebox(0,0)[cc]{${}^{\mathfrak{m}_L(j)}$}}
\put(74,9){\makebox(0,0)[cc]{${}^{\eta}$}}
\put(86,8){\makebox(0,0)[cc]{${}^{\mathfrak{m}_L(i)}$}}
\put(98,8){\makebox(0,0)[cc]{${}^{\mathfrak{m}_L(j)}$}}
\end{picture}
\]
\caption{Relative position for $\alpha=\Delta$ and $\beta$ a ramified unit or a uniformizing parameter.}
\end{figure}

In all remaining cases, either stem, $\mathfrak{m}_k(i)$ or $\mathfrak{m}_k(j)$, is an edge located in 
the $k$-vine minimizing the distance to the corresponding stem, either
$\mathfrak{m}_L(i)$ or $\mathfrak{m}_L(j)$, as in Figure 3C.
If these two $k$-vines are the maximal paths joining $\infty$ with $\eta_1$ and
$\eta_2$, respectively, then the stems of the $k$-branches are located as shown
in one of the Figures 7A-C, unless they coincide, and in the latter case,
this common stem is located as in Figure 7D. The ghost stems 
$\mathfrak{m}_L(i)$ and $\mathfrak{m}_L(j)$ are located in the direction of either
arrow in these Figures. The arrows could coincide in Figure 7D, but this does not affect the proof.
 In the first $3$ cases, the unique path
in the tree $\mathfrak{t}(L)$ from $\mathfrak{m}_L(i)$ to $\mathfrak{m}_L(j)$ passes
through one of  the endpoints of each, $\mathfrak{m}_k(i)$ and $\mathfrak{m}_k(j)$, 
so that the distance between them, by a similar argument as before, is
\small
$$-\frac{1}{2}\nu\left( \frac{\lambda^2-\alpha\beta}{4\alpha\beta} \right)+
\frac12\nu\left(\frac{\eta_1^2-\alpha}{4\alpha}\right)+
\frac12\nu\left(\frac{\eta_2^2-\beta}{4\beta}\right)-1=$$
$$-
\frac12\nu\left(\frac{4(\lambda^2-\alpha\beta)}{(\eta_1^2-\alpha)(\eta_2^2-\beta)}\right)-1=t+s-\frac12\nu\Big(4(\lambda^2-\alpha\beta)\Big)=d_f,$$
\normalsize
where, as before,  there is a final $1$ to take care of the distance $1/2$ from
one endpoint of each stem to its center. 
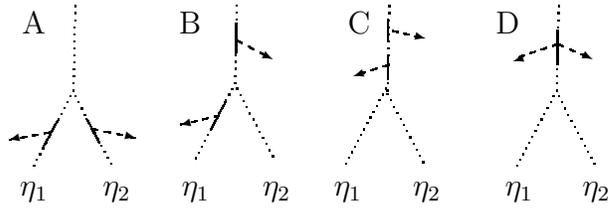
\begin{figure}
\[
\unitlength 1mm 
\linethickness{0.4pt}
\ifx\plotpoint\undefined\newsavebox{\plotpoint}\fi 
\begin{picture}(87.75,28.25)(0,0)
\multiput(18.43,27.18)(-.02083,-.9375){13}{{\rule{.4pt}{.4pt}}}
\multiput(18.18,15.93)(-.40385,-.76923){14}{{\rule{.4pt}{.4pt}}}
\multiput(18.43,15.18)(.43182,-.81818){12}{{\rule{.4pt}{.4pt}}}
\multiput(39.93,28.18)(-.02083,-.9375){13}{{\rule{.4pt}{.4pt}}}
\multiput(39.68,16.93)(-.40385,-.76923){14}{{\rule{.4pt}{.4pt}}}
\multiput(39.93,16.18)(.43182,-.81818){12}{{\rule{.4pt}{.4pt}}}
\multiput(59.68,16.18)(-.40385,-.76923){14}{{\rule{.4pt}{.4pt}}}
\multiput(59.93,15.43)(.43182,-.81818){12}{{\rule{.4pt}{.4pt}}}
\multiput(82.68,27.18)(-.02083,-.9375){13}{{\rule{.4pt}{.4pt}}}
\multiput(82.43,15.93)(-.40385,-.76923){14}{{\rule{.4pt}{.4pt}}}
\multiput(82.68,15.18)(.43182,-.81818){12}{{\rule{.4pt}{.4pt}}}
\multiput(16.1,12)(-.03,-.056){67}{\line(0,-1){.04850746}}
\multiput(20,12.25)(.03365385,-.0625){52}{\line(0,-1){.0625}}
\multiput(36.5,10.75)(.04,.07692308){52}{\line(0,1){.07692308}}
\put(39.9,21){\line(0,1){4}}
\multiput(60.18,27.18)(-.02083,-.9375){13}{{\rule{.4pt}{.4pt}}}
\put(60.1,25.5){\line(0,-1){3}}
\put(60.1,20.5){\line(0,-1){3}}
\put(82.6,24){\line(0,-1){4.5}}
\put(13,2.5){\makebox(0,0)[cc]{$\eta_1$}}
\put(24,2.5){\makebox(0,0)[cc]{$\eta_2$}}
\put(34.5,2.5){\makebox(0,0)[cc]{$\eta_1$}}
\put(45.25,2.5){\makebox(0,0)[cc]{$\eta_2$}}
\put(54.75,2.5){\makebox(0,0)[cc]{$\eta_1$}}
\put(65.25,2.5){\makebox(0,0)[cc]{$\eta_2$}}
\put(77.5,2.5){\makebox(0,0)[cc]{$\eta_1$}}
\put(87.75,2.5){\makebox(0,0)[cc]{$\eta_2$}}
\put(13,25){\makebox(0,0)[cc]{A}}
\put(33.75,25){\makebox(0,0)[cc]{B}}
\put(56.25,25){\makebox(0,0)[cc]{C}}
\put(76,25){\makebox(0,0)[cc]{D}}
\put(26.5,9.75){\vector(4,-1){.07}}\multiput(20.68,10.93)(.136905,-.029762){6}{\line(1,0){.136905}}
\multiput(22.323,10.573)(.136905,-.029762){6}{\line(1,0){.136905}}
\multiput(23.965,10.215)(.136905,-.029762){6}{\line(1,0){.136905}}
\multiput(25.608,9.858)(.136905,-.029762){6}{\line(1,0){.136905}}
\put(32.5,11.5){\vector(-4,-1){.07}}\multiput(37.68,12.68)(-.125,-.029762){6}{\line(-1,0){.125}}
\multiput(36.18,12.323)(-.125,-.029762){6}{\line(-1,0){.125}}
\multiput(34.68,11.965)(-.125,-.029762){6}{\line(-1,0){.125}}
\multiput(33.18,11.608)(-.125,-.029762){6}{\line(-1,0){.125}}
\put(44.75,20.25){\vector(2,-1){.07}}\multiput(39.93,22.68)(.0608974,-.0320513){13}{\line(1,0){.0608974}}
\multiput(41.513,21.846)(.0608974,-.0320513){13}{\line(1,0){.0608974}}
\multiput(43.096,21.013)(.0608974,-.0320513){13}{\line(1,0){.0608974}}
\put(55.5,18){\vector(-3,-1){.07}}\multiput(60.18,19.43)(-.098958,-.03125){8}{\line(-1,0){.098958}}
\multiput(58.596,18.93)(-.098958,-.03125){8}{\line(-1,0){.098958}}
\multiput(57.013,18.43)(-.098958,-.03125){8}{\line(-1,0){.098958}}
\put(65,23){\vector(4,-1){.07}}\multiput(60.68,23.93)(.141667,-.033333){5}{\line(1,0){.141667}}
\multiput(62.096,23.596)(.141667,-.033333){5}{\line(1,0){.141667}}
\multiput(63.513,23.263)(.141667,-.033333){5}{\line(1,0){.141667}}
\put(87.23,20.25){\vector(2,-1){.07}}\multiput(82.68,22.18)(.075,-.033333){10}{\line(1,0){.075}}
\multiput(84.18,21.513)(.075,-.033333){10}{\line(1,0){.075}}
\multiput(85.68,20.846)(.075,-.033333){10}{\line(1,0){.075}}
\put(76.75,20){\vector(-3,-1){.07}}\multiput(82.68,22.18)(-.09375,-.03125){8}{\line(-1,0){.09375}}
\multiput(81.18,21.68)(-.09375,-.03125){8}{\line(-1,0){.09375}}
\multiput(79.68,21.18)(-.09375,-.03125){8}{\line(-1,0){.09375}}
\multiput(78.18,20.68)(-.09375,-.03125){8}{\line(-1,0){.09375}}
\put(9.75,9.5){\vector(-4,-1){.07}}\multiput(15.18,10.43)(-.157143,-.028571){5}{\line(-1,0){.157143}}
\multiput(13.608,10.144)(-.157143,-.028571){5}{\line(-1,0){.157143}}
\multiput(12.037,9.858)(-.157143,-.028571){5}{\line(-1,0){.157143}}
\multiput(10.465,9.573)(-.157143,-.028571){5}{\line(-1,0){.157143}}
\end{picture}
\] 
\caption{Possible locations of the stems, when each element in $\{\alpha,\beta\}$ is either a
ramified unit or a uniformizing parameter. In (D) the arrows might coincide.}
\end{figure}

In the last case, i.e., when the $k$-stems coincide, we claim that
$$t+s-\frac12\nu\Big(4(\lambda^2-\alpha\beta)\Big)=d_f<0.$$
Note that the expression on the left equals
\small
$$-\frac{1}{2}\nu\left( \frac{\lambda^2-\alpha\beta}{4\alpha\beta} \right)+
\frac12\nu\left(\frac{\eta_1^2-\alpha}{4\alpha}\right)+
\frac12\nu\left(\frac{\eta_2^2-\beta}{4\beta}\right)-1,$$
\normalsize
where the first term is the distance between the $L$-stems, while the additive inverses of
the second and third terms are the distance from each of them to the common $k$-stem. 
Since, in this case, the path from one $L$-stem to the other
cannot pass through either endpoint of the $k$-stem, the result follows.\qed  

\section{Computing embedding numbers via gost branches}

In this section we prove Theorem \ref{t22}. For this we make extensive use of Lemma 5.3.

Let us begin by recalling the method employed in \cite{omeat} to compute embedding numbers
for an order $\Omega$ spanning an algebra isomorphic to $k\times k$.
Essentially, it depends on the following observations:
\begin{enumerate}
\item A local Eichler order $\Ea$ of level $n$
 is completely determined by the branch $\mathfrak{s}(\Ea)$,
which is a path of length $n$.
\item The local Eichler order $\Ea$ corresponding to a path $\mathfrak{s}$
contains an order $\Omega$ if and only if $\mathfrak{s}$ is contained in
the branch $\mathfrak{s}(\Omega)$.
\item The stabilizer $\Gamma_1$ of $\Ea$ is the stabilizer of the path $\mathfrak{s}$,
while $\Gamma_2=k^*\Ea^*$ is the point-wise stabilizer of the path, or equivalently,
the stabilizer of either of the walks, $v_0v_1\dots v_n$ or $v_nv_{n-1}\dots v_0$, corresponding
to this path. An element of $\Gamma_2$ that is not
in  $\Gamma_1$ interchanges these two paths.
\item The  orders $\Omega'$, isomorphic to $\Omega$, are in correspondence with
the maximal paths in the BT-tree, while the embeddings $\phi:\Omega\rightarrow\alge$ are in 
correspondence with the idempotents $\tau$, or equivalently, with the oriented maximal
paths, which can be seen as doubly infinite walks.
\item The group of moebius transformation acts transitively on triplets of different ends in 
$\mathbb{P}^1(k)$, while orbits of quartets are fully determined by its cross-ratio.
\item If we have four different vertices $A,B,C,D$ with ends $a,b,c,d$ beyond them as shown in Figure 8,
with given lengths $u,y,s,l,x\in\mathbb{Z}_{\geq0}$,
the orbit of the quartet $(A,B,C,D)$ is determined by the cross-ratio $[a,b;c,d]$ up to
a congruence modulo $\pi^{l+\mathrm{min}\{u,y,s,x\}}\oink$. The same holds for any combination
of ends and vertices by making some of the distances $u$, $y$, $s$ or $x$ infinte.
\end{enumerate}
\begin{figure}
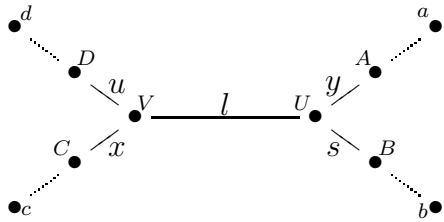

\[  \xygraph{
!{<0cm,0cm>;<.8cm,0cm>:<0cm,.6cm>::} 
!{(16.5,15.3)}*+{l}="l1"
!{(14.7,15.7)}*+{u}="u1"
!{(14.7,14.3)}*+{x}="t1"
!{(18.3,14.3)}*+{s}="s1"
!{(18.3,15.7)}*+{y}="r1"
!{(13,17) }*+{\bullet}="dl"   !{(13.2,17.2) }*+{{}^{d}}="dln"
!{(14,16) }*+{\bullet}="Dl"   !{(14.2,16.2) }*+{{}^{D}}="Dln"
!{(13,13) }*+{\bullet}="cl"   !{(13.2,12.8) }*+{{}^{c}}="cln"
!{(14,14) }*+{\bullet}="Cl"   !{(13.8,14.2) }*+{{}^{C}}="Cln"
!{(15,15) }*+{\bullet}="Vl"   !{(15.2,15.2) }*+{{}^{V}}="Vln"
!{(18,15) }*+{\bullet}="Ul"  !{(17.8,15.2) }*+{{}^{U}}="Uln"
!{(20,17) }*+{\bullet}="al"  !{(19.8,17.2) }*+{{}^{a}}="aln"
!{(19,16) }*+{\bullet}="Al"  !{(18.8,16.2) }*+{{}^{A}}="Aln"
!{(20,13) }*+{\bullet}="bl"  !{(19.8,12.8) }*+{{}^{b}}="bln"
!{(19,14) }*+{\bullet}="Bl"  !{(19.2,14.2) }*+{{}^{B}}="Bln"
"Al"-"Ul" "Ul"-"Bl" "Ul"-"Vl" "Vl"-"Dl" "Vl"-"Cl"
"al"-@{.}"Al" "bl"-@{.}"Bl" "cl"-@{.}"Cl" "dl"-@{.}"Dl"
} 
\]
\caption{ The minimal graph containing four different given vertices, and ends beyond them.
Some lines, like $u$ and $l$ might have length $0$, as long as the vertices $A$, $B$, $C$ and $D$
remain different. Lengths are normalized.} 
\label{figmob}
\end{figure}
To compute the embedding numbers, we first compute the number of walks $v_0\cdots v_r$,
 of length $r$  inside $\mathfrak{b}=\mathfrak{s}_k(\Omega)$,
starting from a given optimal vertex $v_0$. We claim that this number equals $n=q^{[r/2]}$
in our case.
In fact, if for any $i\in\{1,2,\dots,r\}$ we have $p_\mathfrak{b}(v_i)<p_\mathfrak{b}(v_{i-1})$,
the same must hold for every subsequent index and the length $r$ of the walk cannot
exceed $i+p_\mathfrak{b}(v_i)$. Since $v_0$ is optimal, $p_\mathfrak{b}(v_i)\leq i$,
whence $i\geq\frac r2$. On the other hand  $p_\mathfrak{b}(v_i)\leq p(\mathfrak{b})=t$,
whence $i\geq r-t$.
 The smallest value of $i$ satirfying these inequalities is called the returning
point $r_0$. In fact, $r_0=r-[r/2]$ in each case in our setting, since $r\leq 2t+1$.
Since every vertex, in a thick path whose stem has length $0$ or $1$, has at most one neighbor
whose depth is not smaller, the vertices $v_1$ through $v_{r_0}$ are completely determined
by $v_0$. Each subsequent vertex $v_{r_0+1},\dots,v_r$ can be chosen among $q$ different
choices, as every path of length $r-r_0$ from a vertex at depth  
$p_\mathfrak{b}(v_{r_0})=r_0\leq r-r_0$ is completely contained in the branch $\mathfrak{b}$.
The claim follows. 

Similarly, the invariant $m$ denotes, in each case, the number of these walks
for which $v_r$ is also optimal. 
If $r$ is odd, then $m=0$, unless we can have a path of length $r$ with two optimal endpoints,
which necesarily contains a stem edge. This is only possible in the last case in the table.
If $r$ is the diameter of $\mathfrak{s}_k\big(\phi(\Omega)\big)$, then $m=p^{[r/2]}$.
Otherwise, $v_r$ is optimal if and only if $p_\mathfrak{b}(v_{r_0+1})<p_\mathfrak{b}(v_{r_0})$,
and the proportion of paths satisfying this inequality is precisely $\frac{q-1}q$, since there
is exactly one choice of $v_{r_0+1}$ that fails to satisfy it.

Every embedding $\phi:\Omega\rightarrow \matrici_2(k)$ corresponds to a unique
walk $w$ in $\mathfrak{t}(L)$ whose corresponding path is the ghost stem 
$\mathfrak{m}_L\big(\phi(\Omega)\big)$. The ends of this path must be a
$\mathrm{Gal}(L/k)$-orbit,
 for $\phi$ to be defined over $k$.
 Every Eichler order corresponds to a unique finite path in 
$\mathfrak{t}(k)$, or equivalently, two walks. We conclude that
every pair $(\phi,\Ea)$, where $\phi:\Omega\rightarrow \matrici_2(k)$ is an embedding, and 
$\Ea$ is an Eichler order optimally containing $\phi(\Omega)$, corresponds to two pairs
of walks $(w,u_1)$ and $(w,u_2)$ where $u_1$ and $u_2$ define the same path.
Furthermore, the initial vertex of either $u_1$ or $u_2$ is optimal in the corresponding branch
$\mathfrak{s}_k\Big(\phi(\Omega)\Big)$. Recall that $\Gamma_1$ is the stabilizer of the walk $u_1$,
while $\Gamma_2$ is the stabilizer of the pair $\{u_1,u_2\}$. The group of Moebius 
transformations acts transitively on triplets of elements in $k$, whence the orbit of
a triplet $(v,a,b)$, where $v$ is a vertex, while $a$ and $b$ are ends of the BT-tree,
is completely determined by the distance from $v$ to the path joining $a$ and $b$.
A vertex $v_0$ is optimal in the branch $\mathfrak{s}_k\Big(\phi(\Omega)\Big)$,
if and only if its distance to the maximal path corresponding to $\phi$ equals $t$.
This information is given by the invariant of the pair of walks. In fact, if $u_1$ is the path
from $C$ to $D$ in Figure 8, the distance $l$ from $V$ to $U$ in given by the formula
$l=\nu(t-1)$, where $t=[a,b;c,d]$.
 This is immediate, since we can assume $(a,b,c,d)=(\infty,0,1,t)$.  

To compute $e_1$, we fix a walk $u_1$ corresponding to the Eichler order $\Ea$. 
Assume $u_1$ is the walk from $A$ to $B$ in Figure 8.
Observe that
$\Gamma_1$ is the stabilizer of this path. Two optimal embeddings of $\Omega$ into $\Ea$,
corresponding to the doubly infinite walks $w$ and $w'$ from $c$ to $d$, and from $c'$ to $d'$, respectively,
are conjugate by an element stabilizing the path $u_1$ if and only if the invariant is the same,
i.e., $$[a,b;c,d]\equiv[a,b;c',d']\ \left(\mathrm{mod}\ \pi_L^{e(l+\mathrm{min}\{s,y\})}\right),$$ as in this case 
$u=x=\infty$.
Note that both $\{c,d\}$ and $\{c',d'\}$ are $\mathrm{Gal}(L/k)$-orbits, while $a$ and $b$ can be chosen in $k$.
Assume $s\leq y$. Set $b'=\mu(b)$, while $\mu$ is the Moebius transformation satisfying $\mu(a,c,d)=(a,c',d')$,
which is deffined over $k$ by Lemma \ref{rep3}. Then by \cite[Cor. 5.1]{omeat}, $b'$ is an end beyond $B$ in
Figure $8$. Now $\mu$ leaves $u_1$ invariant,
while sends $w$ to $w'$, and is defined over $k$, again by Lemma \ref{rep3}. 
 We conclude that, to compute the number of orbits, we just need
to compute the number of possible invariants. 
Note however that we need to consider embeddings whose corresponding walk has an initial optimal vertex,
and also embeddings whose corresponding walk has an final optimal vertex. This leads us to the formula
$e_1=2n-m=2p^{[r/2]}-m$.

To compute $e_2$, we replace the stabilizer $\Gamma_1$ of the walk $u_1$ by the stabilizer
$\Gamma_2$ of the corresponding path. Note that $\Gamma_1$ is normal in $\Gamma_2$,
and the corresponding quotient acts on the set of $\Gamma_1$-orbits, so we have
$e_2=\frac12(e_1+\chi_2)$ as soon as we prove that $\chi_2$ is the number of invariant orbits.
The orbit of an embedding $\phi$ is invariant if there exists two elements
$\sigma\in\Gamma_1$ and $\lambda\in\Gamma_2\backslash\Gamma_1$
satisfying $\sigma\phi(\omega)\sigma^{-1}=\lambda\phi(\omega)\lambda^{-1}$
for every $\omega\in\Omega$. This is equivalent to the existence of a Moebius transformation
fliping the ends of the path $\mathfrak{s}_k(\Ea)$, while leaving the ends of the branch
$\mathfrak{m}_L\Big(\phi(\Omega)\Big)$ invariant. This is only possible if the following
conditions are satisfied:
\begin{enumerate}
\item Both endpoints of the path $\mathfrak{s}_k(\Ea)$ are equidistant to the ghost stem
 $\mathfrak{m}_L\Big(\phi(\Omega)\Big)$.
\item The invariant $[a,b;c,d]$ of the quartet is its own inverse in the quotient ring
in equation (\ref{999}) (c.f. \S2).
\end{enumerate}
The computation of $e_3=\frac12(e_1+\chi_3)$ is similar, but we no longer require
condition (1) above, as we just need to flip the ends of the infinite path.

To compute $e_4$ we need to compute the number of orbits of pairs of walks, as before,
under an action of the Klein group $C_2\times C_2$ that reverses either walk. This can be done 
via Burnside's Counting Lemma \cite[\S26.10]{r15}. We already know the number of invariants that remain invariant
when we reverse either walk, they are $\chi_2$ and $\chi_3$ respectively.
If we reverse both walks simultaneously, every walk with two optimal endpoints in invariant
since the cross ratio  has the symmetry $[a,b;c,d]=[b,a;d,c]$.  We conclude that
$$e_4=\frac14\Big((2n-m)+\chi_2+\chi_3+m\Big)=\frac n2+\frac14(\chi_2+\chi_3).$$

 If $r=0$ then $\Gamma_1=\Gamma_2$, which impplies $e_1=e_2$ and $e_3=e_4$.
Furthermore $e_1\geq e_3$ as conjugate embeddings have conjugate images.
It suffices therefore to see that $e_1=1$. By another application of  Lemma \ref{rep3},
  the group of Moebius transformations
acts transitively on triples  $(v,a,b)$, where $a$ and $b$ are ends of the BT-tree $\mathfrak{t}(L)$  in the
same Galois orbit, while $v\in V_{\mathfrak{t}(k)}$ is a vertex at a fixed distance, as above.
The result follows. \qed

\end{document}